\documentclass[a4paper,11pt]{amsart}

\pdfpkmode={lexmarkr}
\pdfmapfile{}

\usepackage{amsmath}
\usepackage{amsthm}
\usepackage{amssymb}
\usepackage{xspace}
\usepackage{color}
\usepackage{enumitem}
\usepackage{mathrsfs}
\usepackage{bm}
\usepackage{mathtools}
\usepackage{tikz}

\theoremstyle{plain}
\newtheorem*{definition}{Definition}
\newtheorem{theorem}{Theorem}
\newtheorem{lemma}[theorem]{Lemma}
\newtheorem{corollary}[theorem]{Corollary}
\newtheorem{proposition}[theorem]{Proposition}

\theoremstyle{definition}
\newtheorem{example}[theorem]{Example}
\newtheorem*{exampleagain}{Example~\ref{example}}

\newcommand{\comment}[1]{}

\newcommand{\rea}{\ensuremath{\mathbf{R}}\xspace}
\newcommand{\nat}{\ensuremath{\mathbf{N}}\xspace}
\newcommand{\ints}{\ensuremath{\mathbf{Z}}\xspace}
\newcommand{\intd}[1]{\,\mathrm{d}#1 \,}
\newcommand{\floor}[1]{\left\lfloor #1 \right\rfloor}
\newcommand{\ceil}[1]{\left\lceil #1 \right\rceil}

\newcommand{\leb}{\ensuremath{\mathcal{L}}\xspace}
\newcommand{\Q}{\ensuremath{\mathbf{Q}}\xspace}
\newcommand{\hmeas}{\ensuremath{\mathcal{H}}\xspace}
\newcommand{\ldim}[2]{\ensuremath{\underline{d}_{#1}\!\left(#2\right)}\xspace}
\newcommand{\ldimm}[1]{\ensuremath{\underline{d}_{#1}}\xspace}
\newcommand{\udim}[2]{\ensuremath{\overline{d}_{#1}\!\left(#2\right)}\xspace}
\newcommand{\udimm}[1]{\ensuremath{\overline{d}_{#1}}\xspace}
\newcommand{\anc}{\ensuremath{\preccurlyeq}\xspace}
\newcommand{\itemref}[1]{{\it\ref{#1}}\xspace}

\newcommand{\seq}[1]{\ensuremath{\underline{#1}}\xspace}

\renewcommand{\P}[1]{\ensuremath{\mathbf{P}_{\!#1}}\xspace}
\newcommand{\oball}[2]{\ensuremath{B{\left(#1, #2\right)}}\xspace}
\newcommand{\overbar}[1]{{\mkern 3mu\overline{\mkern-3mu#1\mkern-1mu}\mkern 1mu}}
\newcommand{\cball}[2]{\ensuremath{\overbar{B}{\left(#1, #2\right)}}\xspace}

\newcommand\restr[2]{{
  \left.\kern-\nulldelimiterspace
  #1
  \vphantom{|}
  \right|_{#2}
  }}

\DeclareMathOperator{\supp}{supp}
\DeclareMathOperator{\dist}{dist}
\DeclareMathOperator{\dimh}{dim_H}
\DeclareMathOperator{\dimp}{dim_P}
\DeclareMathOperator{\ldimh}{\underline{dim}_H}
\DeclareMathOperator{\udimh}{\overline{dim}_H}

\DeclareMathOperator{\udimp}{\overline{dim}_P}
\DeclareMathOperator{\udimb}{\overline{dim}_B}
\DeclareMathOperator*{\essinf}{ess\,inf}
\DeclareMathOperator*{\esssup}{ess\,sup}
\DeclareMathOperator*{\E}{\mathbf{E}}

\title{Hausdorff dimension of random limsup sets}

\author{Fredrik Ekstr\"om} \address{Fredrik Ekstr\"om\\ Centre for
  Mathematical Sciences\\ Lund University\\ Box 118\\ 22100
  Lund\\ Sweden} \email{fredrike@maths.lth.se}
\author{Tomas Persson} \address{Tomas Persson\\ Centre for
  Mathematical Sciences\\ Lund University\\ Box 118\\ 22100
  Lund\\ Sweden} \email{tomasp@maths.lth.se}

\begin{document}

\thanks{%
We thank Micha\l{} Rams for explaning things for us. T.~Persson
acknowledges support of the ESI programmes ``Mixing
Flows and Averaging Methods'' and ``Normal Numbers: Arithmetic, Computational
and Probabilistic Aspects'', during which part of this work was done.
}

\begin{abstract}
We prove bounds for the almost sure value of the Hausdorff
dimension of the limsup set of a sequence of balls in $\mathbf{R}^d$
whose centres are independent, identically distributed
random variables. The formulas obtained involve the rate of
decrease of the radii of the balls and multifractal properties
of the measure according to which the balls are distributed,
and generalise formulas that are known to hold for particular
classes of measures.
\end{abstract}


\maketitle

\section{Introduction}
Suppose we are given a sequence $(l_n)$ of positive numbers and that
we randomly toss out arcs of length $l_n$ on a circle of circumference
1. If the centres of the arcs are independent and uniformly
distributed, then by the Borel--Cantelli lemma, almost surely almost
all points of the circle are covered infinitely often if and only if
$\sum l_n$ diverges. In this case, A.~Dvoretzky \cite{Dvoretzky} asked
the question when almost surely all points of the circle are covered
infinitely often. After several partial results, the question
was finally answered by L.~Shepp \cite{Shepp}, who proved that with
probability 1, all points are covered infinitely many times if and
only if $\sum n^{-2} \exp (l_1 + \cdots + l_n)$ diverges.

In case $\sum l_n$ is finite, the set of points that are covered
infinitely often is always of zero Lebesgue measure, but may be
large in some other sense, for instance Hausdorff dimension. A.-H.~Fan
and J.~Wu \cite{FanWu} studied the case $l_n = n^{-\alpha}$, where
$\alpha > 1$ is a fixed number, and proved that the Hausdorff
dimension of the set of points covered infinitely often is almost
surely equal to $1/\alpha$. A.~Durand \cite{Durand} considered
arbitrary lengths $l_n$ and showed that the Hausdorff dimension is
almost surely given by the infimum of those $t$ for which $\sum l_n^t$
converges. Moreover, for this $t$, Durand showed that almost surely
the set of points covered infinitely often has what is called large
intersections. This means that the set belongs to a class of
$G_\delta$ sets of Hausdorff dimension at least $t$, which is closed
under bi-Lipschitz mappings and countable intersections.

E.~J\"arvenp\"a\"a, M.~J\"arvenp\"a\"a, H.~Koivusalo, B.~Li and
V. Suomala \cite{Jarvenpaaetal} considered a $d$-dimensional torus and
a sequence of ellipsoids satisfying certain regularity
conditions. They showed that if the centres of the ellipsoids are
chosen randomly and uniformly distributed on the torus, then the
Hausdorff dimension of the set of points covered infinitely often is
almost surely given by a certain formula. Using potential-theoretic
arguments, their result was generalised by T.~Persson \cite{Persson}
to the case when the torus is covered randomly with a sequence of
arbitrary open sets. In this case the almost sure value of the
Hausdorff dimension was estimated from below, an estimate that can
easily be seen to also be an upper estimate in case the sequence of
open sets are ellipsoids.

D.-J.~Feng, E.~J\"arvenp\"a\"a, M.~J\"arvenp\"a\"a, V.~Suomala
\cite{Fenetal} improved on the result by Persson, by considering the
cover of a Riemannian manifold with open sets that are independently
distributed according to a non-singular measure. They showed that the
estimate from below given by Persson also holds in this case, and that
it gives the correct result if one considers instead the supremum of
all such values obtained by replacing the open sets by their subsets.

In this paper we shall consider an arbitrary Borel probability measure
on $\mathbf{R}^d$, and a sequence of positive numbers $(r_n)$. We are
interested in the almost sure value of the Hausdorff dimension of the
set of points covered infinitely often by balls of radius $r_n$ for
which the centres are chosen independently and distributed according
to the measure $\mu$. This problem was also considered by S.~Seuret
\cite{Seuret}, but he restricted to the case when $\mu$ is a Gibbs
measure invariant under an expanding Markov map. Because of this
assumption, it is possible to use techniques from thermodynamic
formalism, and with such tools Seuret gave the almost sure value of
the Hausdorff dimension in terms of multifractal spectra of the
measure $\mu$.  We will assume nothing about the measure $\mu$ and
give estimates on the Hausdorff dimension from below and above. Our
estimates are, as the formulas by Seuret, in terms of multifractal
properties of $\mu$. Under some assumptions on $\mu$, which are
weaker than the assumption used by Seuret, our estimates from below
and above coincide, and we get a formula for the almost sure value of
the Hausdorff dimension.

\section{Results}
Let $\mu$ be a Borel probability measure on $\rea^d$ and consider the
probability space $(\Omega, \P{\mu})$, where $\Omega = (\rea^d)^\nat$ and
$\P{\mu} = \mu^\nat$. For $\alpha > 0$ and $\omega$ in $\Omega$, let
\[
E_\alpha(\omega) = \limsup_n \cball{\omega_n}{n^{-\alpha}} =
\bigcap_{k=1}^\infty \bigcup_{n=k}^\infty
\cball{\omega_n}{n^{-\alpha}}.
\]
For every real number $s$, the event $\{\hmeas^s(E_\alpha) = \infty\}$
is a tail event with respect to $\P{\mu}$, and thus $\dimh E_\alpha$ is
$\P{\mu}$-a.s.~constant by Kolmogorov's zero-one law
\cite[Theorem~IV.1.1]{Shiryaev}.  Let
\[
f_\mu(\alpha) = \P{\mu}\text{-a.s.~value of } \dimh E_\alpha.
\]
By using families of the form $\{\cball{\omega_n}{n^{-\alpha}}\}_{n = n_0}^\infty$
to cover $E_\alpha(\omega)$, it is not so difficult to see that
$\dimh E_\alpha(\omega) \leq 1 / \alpha$ for every $\omega$, so that
$f_\mu(\alpha) \leq 1 / \alpha$.

The \emph{lower local dimension} and \emph{upper local dimension}
of $\mu$ at a point $x$ in $\rea^d$ are defined by
	\begin{align*}
	\ldim{\mu}{x} &= \sup \{ s; \, \exists c \text{ s.t. } \mu(\oball{x}{r}) \leq
	cr^s \text{ for all } r > 0\},
	\qquad\text{and}\\
	\udim{\mu}{x} &= \inf \{ s; \, \exists c \text{ s.t. } \mu(\oball{x}{r}) \geq
	cr^s \text{ for all } r > 0\},
	\end{align*}
respectively. Define the \emph{local dimension discrepancy} of $\mu$ at $x$ to be
	\[
	\delta_\mu(x) = \udim{\mu}{x} - \ldim{\mu}{x}
	\]
(this is only defined $\mu$-a.e., namely, at those points where at least one
of $\udimm{\mu}$ and $\ldimm{\mu}$ is finite).

\begin{theorem} \label{mainthm}
If\/ $1 / \alpha < \udimh \mu$ then $ f_\mu(\alpha) \geq 1 / \alpha - \delta$,
where
	\[
	\delta = \essinf_{\substack{x \sim \mu \\ \ldim{\mu}{x} > 1 /
	\alpha}} \delta_\mu(x).
	\]
\end{theorem}

Theorem~\ref{mainthm} will be proved at the end of
Section~\ref{sec:proofmain}.  Preparations needed for the proof are
done in Sections~\mbox{\ref{sec:energy}--\ref{sec:proofmain}}.

If $A \subset \rea$ and $h \colon A \to \rea$, we say that $h$ is
1-Lipschitz continuous if $|h(x) - h(y)| \leq |x - y|$ for all $x, y
\in A$. If $g \colon A \to \rea$, let $\bar{g}$ denote the
\emph{increasing 1-Lipschitz hull} of $g$, defined by
	\[
	\bar{g} (x) = \inf \{ h(x) ; \,
	h \text{ is increasing, 1-Lipschitz continuous, and } h \geq g \}.
	\]
Also, define $\tilde g$ by $\tilde g (x) = x + \sup_{y \geq x} (g(y) - y)$.
Then $\tilde g \leq \bar g$, with equality if $g$ is increasing.

Let
\begin{align*}
  F_\mu (s) &= \dimh \{ x \in \supp \mu ; \, \ldim{\mu}{x}
  \leq s \},\\
  G_\mu (s) &= \lim_{\varepsilon \to 0} \limsup_{r \to 0} \frac{\log
    (N_r (s + \varepsilon) - N_r (s - \varepsilon))}{- \log r},
\end{align*}
where $N_r (s)$ denotes the number of $d$-dimensional cubes of
the form
\[
Q = [k_1 r, (k_1+1) r) \times \ldots \times [k_d r, (k_d+1) r)
\]
with $k_1, \ldots, k_d \in \mathbf{Z}$ and $\mu (Q) \geq r^s$. The
function $F_\mu$ is called the \emph{Hausdorff spectrum} of the lower
local dimension of $\mu$, and $G_\mu$ is called the
\emph{upper coarse spectrum} of $\mu$.

\begin{proposition} \label{pro:main}
For every $\alpha > 0$,
	\[
	\lim_{s \to 1/\alpha^-} F_\mu(s)
	\leq f_\mu(\alpha)
	\leq \max\left(F_\mu(1/\alpha), \, \tilde{G}_\mu (1 / \alpha)\right).
	\]
\end{proposition}

Proposition~\ref{pro:main} will be proved in
Section~\ref{sec:proofprop}. With a proof similar to that of
Proposition~\ref{pro:main}, we can also show the following alternative
upper bound. The \emph{packing dimension} of a set $A \subset \rea^d$
can be expressed as
\begin{equation} \label{dimpeq}
  \dimp A = \inf\left\{t; \, \exists \{A_i\}_{i = 1}^\infty
  \text{ s.t.~}
%	A_i \text{ is compact and }
  \udimb A_i \leq t \text{ for all } i
  \text{ and }
  A \subset \bigcup_{i = 1}^\infty A_i
  \right\}.
\end{equation}
The \emph{packing spectrum} of the lower local dimension of $\mu$ is
defined by
\[
H_\mu (s) = \dimp \{ x ;\, \ldim{\mu}{x} \leq s \}.
\]

\begin{proposition} \label{pro:alternative}
  For every $\alpha > 0$, it holds that $f_\mu(\alpha) \leq \bar{H}_\mu (1 /
  \alpha)$, where $\bar{H}_\mu$ denotes the increasing 1-Lipschitz
  hull of $H_\mu$.
\end{proposition}

The proof is in Section~\ref{sec:alternative}. It is possible to have
a measure for which $F_\mu = H_\mu$, see for instance Example~\ref{example} below.

Our bounds for $f_\mu$ simplify if some
assumptions are made about $\mu$. If $\delta_\mu = 0$ almost
everywhere then $\delta = 0$ in Theorem~\ref{mainthm},
so that $f_\mu(\alpha) \geq 1 / \alpha = \bar F_\mu(1 / \alpha)$ for
$1 / \alpha \leq \udimh \mu$. If $F_\mu$ is
concave then $\bar F_\mu(s) = F_\mu(s)$ for
$s \geq \ldimh \mu$, using that the point $(\ldimh \mu, \ldimh \mu)$
lies in the closure of the graph of $F_\mu$. Both of these
conditions are often satisfied if $\mu$ is a Gibbs measure, for
instance in the setting of Seuret's paper.
Moreover, in that case it is well known that
	\[
	G_\mu(s) = F_\mu(s) = \Psi_\mu(s) := \dimh \{x; \, \ldim{\mu}{x} = s \}
	\]
for $s \leq s_*$, where $s_*$ is the value of $s$ that
maximises $\Psi_\mu(s)$,
and $F_\mu(s) \geq G_\mu(s)$ for $s \geq s_*$ (see~\cite[Proposition 1 and~2]{Seuret}).
Thus in this case, the upper bound in Proposition~\ref{pro:main} is equal
to $\bar F_\mu$, and
	\[
	f_\mu(\alpha) = \bar F_\mu(1 / \alpha)
	\]
for all $\alpha > 0$, that is, our bounds
specialise to Seuret's result (see Figure~\ref{gibbsfig}).

  \begin{figure}[h!]
    \begin{tikzpicture}[scale=3]
      \draw[->,>=latex] (0,0)--(0,1.1) node[right] {$y$};
      \draw[->,>=latex] (0,0)--(1.1,0) node[right] {$s$};
      \node at (0.7,0.6) [label=0:{$y = F_\mu(s)$}] {};
      
      \draw [dashed] (0,0)--(1.0,1.0);

      \draw[blue, thick]
      plot[smooth,domain=0.15:0.6,samples=100]
      ({\x},{0.4 + (\x-0.4) * (0.8-\x) / (0.8-0.4)});
      \draw[blue, thick] (0.6,0.5)--(1,0.5);

      \draw[black]
      plot[smooth,domain=0.6:1.05,samples=100]
      ({\x},{0.4 + (\x-0.4) * (0.8-\x) / (0.8-0.4)});

      \draw [dashed] (0.4,0)--(0.4,0.4);
      \draw (0.4, -0.02)--(0.4, 0.02);
      \node at (0.4,0) [label=-90:{$\dimh \mu$}] {};

    \end{tikzpicture}
    \qquad \qquad
    \begin{tikzpicture}[scale=3]
      \draw[->,>=latex] (0,0)--(0,1.1) node[right] {$y$};
      \draw[->,>=latex] (0,0)--(1.1,0) node[right] {$1/\alpha$};
      \node at (0.7,0.6) [label=0:{$y = f_\mu(\alpha)$}] {};

      \draw [dashed] (0,0)--(1.0,1.0);

      \draw[black]
      plot[smooth,domain=0.15:1.05,samples=100]
      ({\x},{0.4 + (\x-0.4) * (0.8-\x) / (0.8-0.4)});

      \draw[red,thick]
      plot[smooth,domain=0.4:0.6,samples=100]
      ({\x},{0.4 + (\x-0.4) * (0.8-\x) / (0.8-0.4)});
      \draw[red, thick] (0,0)--(0.4,0.4);
      \draw[red, thick] (0.6,0.5)--(1,0.5);

      \draw [dashed] (0.4,0)--(0.4,0.4);
      \draw (0.4, -0.02)--(0.4, 0.02);
      \node at (0.4,0) [label=-90:{$\dimh \mu$}] {};
           
    \end{tikzpicture}
    \caption{The graphs of $F_\mu$ (blue), $f_\mu$ (red) and
      $\Psi_\mu$ (black).}\label{gibbsfig}
  \end{figure}
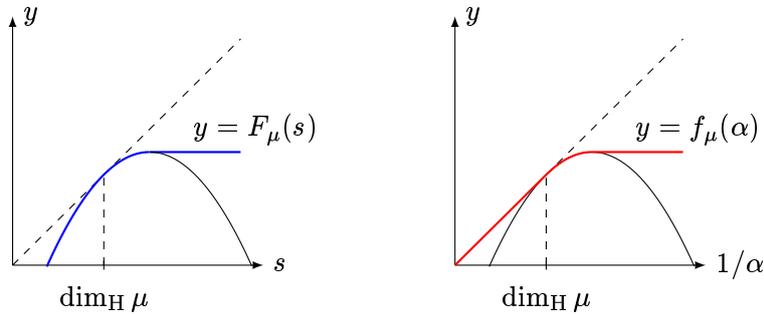

The arguments above give the following corollary.

\begin{corollary}
  Suppose that $\mu$ is a measure such that $F_\mu$ is concave on the
  interval on which it is positive, and that the local dimension
  discrepancy $\delta_\mu$ vanishes almost everywhere. If either
  $F_\mu(s) \geq \tilde{G}_\mu(s)$ or $F_\mu(s) = H_\mu(s)$
for $s \geq \udimh \mu$, then for every $\alpha > 0$,
  \[
  f_\mu (\alpha) = \bar F_\mu (1/\alpha).
  \]
\end{corollary}

\subsection{An example}

One might conjecture that it always holds that $f_\mu(\alpha) = \bar
F_\mu(1 / \alpha)$.  The following example describes a situation where
this holds, and where $\bar F_\mu$ is not concave.

\begin{example} \label{example}
  Take $\lambda \in (0, 1 / 3)$ and let $\nu$ be the Bernoulli
  convolution with parameter $\lambda$. Let $C$ be the ternary Cantor
  set, and for $k = 0, 1, \ldots$ let $C_k$ be the set in the $k$:th
  step of the standard construction of $C$---thus $C_k$ is a union of
  $2^k$ intervals of length $3^{-k}$.  For each $k$, let $\mu_k$ be
  the sum of $2^k$ affine images of $\nu$, located in the middle
  ninths of the intervals of $C_k$. Equivalently, each component of
  $\mu_k$ is concentrated in the middle third of one of the gaps that
  appear when $C_{k + 1}$ is constructed from $C_k$.  The total mass
  of $\mu_k$ is $2^k$.

  Let $\beta > 1$ and define the probability measure
  \[
  \mu = \frac{\sum_{k = 0}^\infty 2^{-\beta k} \mu_k}{\sum_{k =
      0}^\infty 2^{(1 - \beta)k}} = \left(1 - 2^{1 - \beta}\right)
  \sum_{k = 0}^\infty 2^{-\beta k} \mu_k.
  \]
  It will be shown that the lower and upper local dimensions of $\mu$
  agree everywhere, and $F_\mu$ will be computed. The support of $\mu$
  is
  \[
  \supp \mu = C \cup \left( \bigcup_{k = 0}^\infty \supp \mu_k \right).
  \]
  If $x \in \supp \mu_k$ for some $k$ then
  \[
  \ldim{\mu}{x} = \udim{\mu}{x} = \frac{\log 2}{- \log \lambda} =: s_0
  \]
  (this is the dimension of the Bernoulli convolution), and if $x \notin
  \supp \mu$ then $\ldim{\mu}{x} = \udim{\mu}{x} = \infty$.

  Consider some $x \in C$ and $r \in (0, 1)$, and let $n$ be the
  unique integer such that $3^{-(n + 1)} \leq r < 3^{-n}$. Then $[x -
    r, x + r]$ includes the interval of $C_{n + 1}$ that contains $x$,
  and thus $\mu_{n + l}([x - r, x + r]) \geq 2^{l - 1}$ for $l = 1, 2,
  \ldots$.  It follows that
  \[
  \mu ([x - r, x + r]) \geq \left(1 - 2^{1 - \beta}\right) \sum_{l =
    1}^\infty 2^{-\beta(n + l)} 2^{l - 1} = 2^{- \beta(n + 1)}.
  \]
  On the other hand,
  \[
    [x - r, x + r] \cap \supp \mu_k = \emptyset \qquad\text{for } k
    \leq n - 2,
  \]
  and $[x - r, x + r]$ intersects just one interval of $C_{n - 1}$ so
  that $\mu_{n - 1 + l} \leq 2^l$ for $l = 0, 1, \ldots$. It follows
  that
  \[
  \mu ([x - r, x + r]) \leq \left(1 - 2^{1 - \beta}\right) \sum_{l =
    0}^\infty 2^{-\beta(n - 1 + l)} 2^l = 2^{- \beta(n - 1)}.
  \]
  These bounds for the measure of $[x - r, x + r]$ show that
  \[
  \ldim{\mu}{x} = \udim{\mu}{x} = \frac{\beta \log 2}{\log 3} =: s_1
  \]
  for $x \in C$.

  In summary,
  \[
	\ldim{\mu}{x} =
	\begin{cases}
	s_0 & \text{if } x\in \bigcup_{k = 1}^\infty \supp \mu_k \\
	s_1 & \text{if } x \in C \\
	\infty & \text{otherwise},
	\end{cases}
  \]
  which implies that
  \[
	F_\mu(s) =
	\begin{cases}
	0	& \text{if } s < s_0 \\
	s_0	& \text{if } s_0 \leq s < s_1 \\
	d_1	&\text{if } s_1 \leq s,
	\end{cases}
  \]
  where $d_1 = \log 2 / \log 3$. The graphs of $F_\mu$ and $\bar
  F_\mu$ are shown in the Figure~\ref{fig:example}.

  \begin{figure}[ht!]
    \begin{tikzpicture}[scale=3]
      \draw[->,>=latex] (0,0)--(0,1.1) node[right] {$y$};
      \draw[->,>=latex] (0,0)--(1.1,0) node[right] {$s$};
      \node at (0.7,0.5) [label=0:{$y = F_\mu(s)$}] {};

      \draw (0.38685, -0.02)--(0.38685, 0.02);
      \node at (0.38685,0) [label=-90:{$s_0$}] {};
      
      \draw (0.82021, -0.02)--(0.82021, 0.02);
      \node at (0.82021, 0) [label=-90:{$s_1$}] {};
      
      \draw (-0.02, 0.38685)--(0.02, 0.38685);
      \node at (0, 0.38685) [label=180:{$s_0$}] {};

      \draw (-0.02, 0.63093)--(0.02, 0.63093);
      \node at (0, 0.63093) [label=180:{$d_1$}] {};

      \draw [blue, thick] (0, 0)--(0.38685, 0);
      \node at (0.38685, 0) [circle,draw,white,fill,inner sep=0.3mm] {};
      \node at (0.38685, 0) [circle,draw,blue,inner sep=0.3mm] {};
      
      \draw [blue, thick] (0.38685, 0.38685)--(0.82021, 0.38685);
      \node at (0.82021, 0.38685) [circle,draw,white,fill,inner sep=0.3mm,] {};
      \node at (0.82021, 0.38685) [circle,draw,blue,inner sep=0.3mm,] {};
      \node at (0.38685, 0.38685) [circle,draw,blue,fill,inner sep=0.3mm] {};

      \draw [blue, thick] (0.82021, 0.63093)--(1.1, 0.63093);
      \node at (0.82021,0.63093) [circle,draw,blue,fill,inner sep=0.3mm] {};

      \draw [dashed] (0,0)--(1.0,1.0);
    \end{tikzpicture}
    \qquad \qquad
    \begin{tikzpicture}[scale=3]
      \draw[->,>=latex] (0,0)--(0,1.1) node[right] {$y$};
      \draw[->,>=latex] (0,0)--(1.1,0) node[right] {$s$};
      \node at (0.7,0.5) [label=0:{$y = \bar{F}_\mu(s)$}] {};

      \draw (0.38685, -0.02)--(0.38685, 0.02);
      \node at (0.38685,0) [label=-90:{$s_0$}] {};
      
      \draw (0.82021, -0.02)--(0.82021, 0.02);
      \node at (0.82021, 0) [label=-90:{$s_1$}] {};
      
      \draw (-0.02, 0.38685)--(0.02, 0.38685);
      \node at (0, 0.38685) [label=180:{$s_0$}] {};
      
      \draw (-0.02, 0.63093)--(0.02, 0.63093);
      \node at (0, 0.63093) [label=180:{$d_1$}] {};

      \draw [blue, thick] (0, 0)--(0.38685, 0);
      \node at (0.38685, 0) [circle,draw,white,fill,inner sep=0.3mm] {};
      \node at (0.38685, 0) [circle,draw,blue,inner sep=0.3mm] {};

      \draw [blue, thick] (0.38685, 0.38685)--(0.82021, 0.38685);
      \node at (0.82021, 0.38685) [circle,draw,white,fill,inner sep=0.3mm,] {};
      \node at (0.82021, 0.38685) [circle,draw,blue,inner sep=0.3mm,] {};
      \node at (0.38685, 0.38685) [circle,draw,blue,fill,inner sep=0.3mm] {};

      \draw [blue, thick] (0.82021, 0.63093)--(1.1, 0.63093);
      \node at (0.82021,0.63093) [circle,draw,blue,fill,inner sep=0.3mm] {};
      
      \draw [dashed] (0,0)--(1.0,1.0);
      
      \draw [red, thick] (0,0)--(0.38685, 0.38685)--(0.57613, 0.38685)--(0.82021, 0.63093)-- (1.1, 0.63093);
    \end{tikzpicture}
    \caption{The graphs of $F_\mu$ (blue) and $\bar F_\mu$ (red).} \label{fig:example}
  \end{figure}
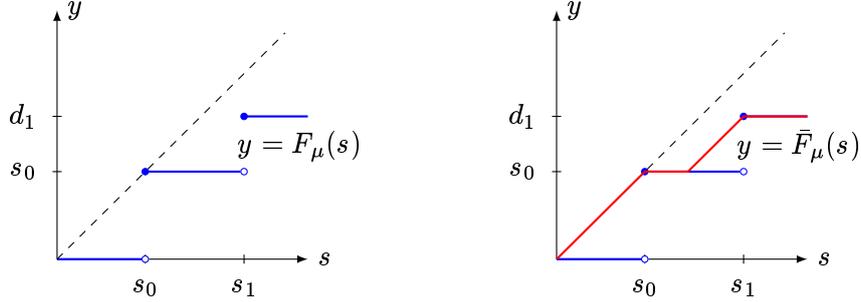

  Since the lower and upper local dimensions of $\mu$ agree
  everywhere, the discrepancy of $\mu$ is $0$, and thus
  Theorem~\ref{mainthm} implies that $f_\mu(\alpha) = \bar F_\mu(1 /
  \alpha)$ for $1 / \alpha \leq s_0$. Proposition \ref{pro:main} and ~\ref{pro:alternative}
  show that $F_\mu(1 / \alpha-) \leq f_\mu(\alpha) \leq \bar F_\mu(1
  / \alpha)$, using that $F_\mu = H_\mu$ for the measure
  $\mu$, since the sub-level sets of $\ldimm{\mu}$ are regular Cantor sets
and thus have equal Hausdorff and packing dimension.
It remains to show that
  \[
  f_\mu(\alpha) \geq \frac{1}{\alpha} - (s_1 - d_1) = \frac{1}{\alpha}
  - (\beta - 1) \frac{\log 2}{\log 3}
  \]
  for $1 / \alpha \leq s_1$. This will be done in
  Section~\ref{sec:example}.
\end{example}

\subsection{Some conventions, definitions and facts} \label{ssec:notation}
By the natural numbers we mean the set $\{1, 2, 3, \ldots\}$,
and sequences will generally be indexed starting from $1$.
All measures on topological spaces below are Borel measures,
but this will not be explicitly stated each time. Thus
``measure'' below should be interpreted as ``Borel measure''
whenever such an interpretation makes sense.

The open and closed balls with centre $x$ and radius $r$ are
denoted by $\oball{x}{r}$ and $\cball{x}{r}$ respectively.
If $B$ is a ball (open or closed), then $x(B)$ denotes its
centre and $r(B)$ its radius.

Let $\mu$ be a measure on $\rea^d$.
The \emph{upper Hausdorff dimension} and \emph{upper packing dimension} of
$\mu$ are defined by
	\begin{align*}
	\udimh \mu &= \esssup_{x \sim \mu} \ldim{\mu}{x}, \qquad\text{and} \\
	\udimp \mu &= \esssup_{x \sim \mu} \udim{\mu}{x},
	\end{align*}
respectively. If $A$ is a set of full $\mu$-measure, then
$\dimh A \geq \udimh \mu$. The measure $\mu$ is \emph{$(c, s)$-uniform} if
$\mu$-a.e.~$x$ is such that
	\[
	\mu(\cball{x}{r}) \leq cr^s
	\qquad\text{for every } r > 0,
	\]
and $\mu$ is \emph{$(c, s)'$-uniform} if $\mu$-a.e.~$x$ is such that
	\[
	\mu\left(\cball{x}{r} \setminus \{x\}\right) \leq cr^s
	\qquad\text{for every } r > 0.
	\]
If $\mu$ is $(c, s)$-uniform then every Borel set $A$ satisfies
$\mu(A) \leq c |A|^s$, where $|A|$ denotes the diameter of $A$.
Any non-atomic measure that is $(c, s)'$-uniform is also
$(c, s)$-uniform.  We will use two variants of the $t$-energy of
$\mu$, namely,
\begin{align*}
  I_t(\mu) &= \iint |x - y|^{-t} \intd\mu(x)\intd\mu(y), \qquad\text{and}\\
  I'_t(\mu) &= \iint_{x \neq y} |x - y|^{-t} \intd\mu(x)\intd\mu(y).
\end{align*}
For non-atomic measures they are the same, but $I'_t$ gives some
information about discrete measures as well. It is well known that
if $I_t(\mu) < 0$ then every set of positive $\mu$-measure has
Hausdorff dimension greater than or equal to $t$.

\section{A lemma about the energy integral}	\label{sec:energy}
If $\mu$ is a $\sigma$-finite measure on $\rea^d$ and $\varphi$ is
a non-negative measurable function on $\rea^d$, then
	\[
	\int \varphi \intd\mu = \int_0^\infty \mu\left\{
	y \in \rea^d; \, \varphi(y) \geq z
	\right\} \intd z,
	\]
since by Fubini's theorem both sides are equal to the $\mu \times \leb$-measure
of the set $\{(y, z) \in \rea^d \times [0, \infty]; \, z \leq \varphi(y) \}$.
In particular, for $t > 0$ and $x$ in $\rea^d$,
	\begin{equation} \label{poteq}
	\int_{x \neq y} |x - y|^{-t} \intd\mu(y) =
	\int_0^\infty \mu\left(\cball{x}{z^{-1/t}} \setminus \{x\} \right)\intd z.
	\end{equation}

\begin{lemma} \label{Itlemma}
Let $\mu$ be a $(c, s)'$-uniform measure on $\rea^d$ and let $t \in (0, s)$.
Then for every Borel subset $A$ of\/ $\rea^d$,
	\[
	I'_t\left(\restr{\mu}{A}\right) \leq \frac{cs}{s - t}|A|^{s - t}\mu(A).
	\]

\begin{proof}
Applying~\eqref{poteq} to $\restr{\mu}{A}$ gives for $\restr{\mu}{A}$-a.e.~$x$,
	\begin{align*}
	\int_{x \neq y} |x - y|^{-t} \intd{\restr{\mu}{A}}(y) &=
	\int_0^\infty \restr{\mu}{A}\left(\cball{x}{z^{-1/t}} \setminus \{x\} \right)\intd z \\
	&\leq
	c \int_0^\infty
	\min\left(
	|A|^s, \, z^{-s/t}
	\right)\intd z
	=
	\frac{cs}{s - t}|A|^{s - t}.
	\end{align*}
Integrating this over $x$ with respect to $\restr{\mu}{A}$ proves the lemma.
\end{proof}
\end{lemma}

\section{Fractal trees}	\label{treesec}
\begin{definition}
A \emph{fractal tree} is a triple $(\mathcal B, R, \pi)$ such that
\begin{enumerate}[label=\roman*)]
\item
$\mathcal B$ is a set of closed balls in $\rea^d$, and $R \in \mathcal B$,

\item
$\pi$ is a function $\mathcal B \setminus \{R\} \to \mathcal B$ such
that for every $B \in \mathcal B$ there exists $g \in \{0, 1, 2, \ldots\}$
such that $\pi^g(B) = R$, and

\item \label{treepropiii}
for every $B \in \mathcal B$, the balls in $\pi^{-1}(B)$ have the same radius,
are disjoint and included in $B$, and $2 \leq \#\pi^{-1}(B) < \infty$.
\end{enumerate}
\end{definition}

The ball $R$ is the \emph{root} of the tree, $\pi(B)$ is the
\emph{parent} of $B$ and $\pi^{-1}(B)$ are
the \emph{children} of $B$. The number $g$ in \emph{ii)} is uniquely
determined since $R$ does not have a parent, and thus defines
a function $g \colon \mathcal B \to \{0, 1, 2, \ldots\}$. The number
$g(B)$ is the \emph{generation} of $B$. From \emph{iii)} it follows
that there are only finitely many balls of any given generation,
and in particular $\mathcal B$ is countable. Property \emph{iii)}
also implies that $r(B) \leq D 2^{-g(B)}$, where $D$ is the
diameter of the balls of generation $1$. A ball $A \in \mathcal B$
is an \emph{ancestor} of $B$, denoted $A \anc B$, if there is some
$k \in \{0, 1, 2, \ldots\}$ such that $\pi^k(B) = A$.

Associated to a fractal tree $(\mathcal B, R, \pi)$ is the set
$K(\mathcal B) = \limsup \mathcal B$, consisting of those points
that are contained in infinitely many balls in $\mathcal B$.
The set $K(\mathcal B)$ is totally disconnected and perfect,
and can also be expressed as a decreasing intersection of compact
sets by
	\[
	K(\mathcal B) = \bigcap_{k = 0}^\infty \Big(
%	\bigcup_{\substack{B \in \mathcal B \\ g(B) = k}}
	\bigcup_{g(B) = k}
	B\Big).
	\]

A probability measure $\theta$ that is concentrated on $K(\mathcal B)$
can be defined as follows. For $B \in \mathcal B$ let $\nu_B$ be the uniform
probability measure on the centres of the children of $B$, that is
	\[
	\nu_B = \frac{1}{\# \pi^{-1}(B)} \sum_{C \in \pi^{-1}(B)} \delta_{x(C)},
	\]
and define
	\[
	\theta(B) = \prod_{R \neq A \anc B} \nu_{\pi(A)}(A).
	\]
Then $\theta(R) = 1$, and
	\[
	\theta(B) = \sum_{C \in \pi^{-1}(B)} \theta(C)
	\]
for every $B \in \mathcal B$.
It follows that $\theta$ can be extended to an additive measure on the algebra
$\mathcal A$ on $K(\mathcal B)$ consisting of the intersections of $K(\mathcal B)$ with
finite unions of balls in $\mathcal B$.
Since the sets in $\mathcal A$ are both compact and open in the relative
topology on $K(\mathcal B)$, it is not possible to express a set in
$\mathcal A$ as a union of infinitely many disjoint sets in $\mathcal A$.
Hence $\theta$ is vacuously $\sigma$-additive on $\mathcal A$, and
can therefore be extended to a probability measure on $K(\mathcal B)$
by Carath\'eodory's theorem~\cite[Section~II.3]{Shiryaev}.

\begin{lemma}	\label{treelemma}
Let $(\mathcal B, R, \pi)$ be a fractal tree such that whenever
$B_1, B_2 \in \mathcal B$ and $\pi(B_1) = \pi(B_2)$ then
	\[
	\dist\left(B_1, B_2\right) \geq \frac{1}{2} \left|x(B_1) - x(B_2)\right|.
	\]
Let $s$, $t$ and $\{c_B\}_{B \in \mathcal B}$ be positive numbers such that
$t < s$ and $\nu_B$ is $(c_B, s)'$-uniform for every $B \in \mathcal B$. Then
	\[
	I_t(\theta) \leq \frac{2^t s}{s - t}
	\sum_{B \in \mathcal B}\theta(B)^2 c_B |B|^{s - t}.
	\]

\begin{proof}
Since every ball in $\mathcal B$
has at least two children, the measure $\theta$ is non-atomic. If $x_1$ and
$x_2$ are distinct points in $K(\mathcal B)$ then there exist uniquely
a ball $B$ in $\mathcal B$ and distinct children $C_1$ and $C_2$ of $B$
such that $x_1 \in C_1$ and $x_2 \in C_2$. Thus
	\begin{align*}
	I_t(\theta) = I_t'(\theta)
	&=
	\iint_{x_1 \neq x_2} |x_1 - x_2|^{-t} \intd\theta(x_1)\intd\theta(x_2) \\
	&=
	\sum_{B \in \mathcal B} {\sum}_B'
%	\sum_{\substack{C_1, C_2 \in \pi^{-1}(B) \\ C_1 \neq C_2}}
	\iint_{C_1 \times C_2} |x_1 - x_2|^{-t} \intd\theta(x_1)\intd\theta(x_2),
	\end{align*}
where $\sum_B'$ denotes the sum over all pairs $(C_1, C_2)$ of distinct
children of $B$. The integral can be estimated by
	\[
	\iint_{C_1 \times C_2} |x_1 - x_2|^{-t} \intd\theta(x_1)\intd\theta(x_2) \leq
	2^t \left|x(C_1) - x(C_2)\right|^{-t} \theta(C_1)\theta(C_2),
	\]
and hence, since $\theta(C_i) = \theta(B) \nu_B(x(C_i))$, the inner sum is less than
or equal to
	\begin{align*}
	2^t \theta(B)^2 {\sum}_B'
%	\sum_{\substack{C_1, C_2 \in \pi^{-1}(B) \\ C_1 \neq C_2}}
	|x(C_1) - x(C_2)|^{-t} \nu_B(x(C_1)) \nu_B(x(C_2))
	&= 2^t \theta(B)^2 I'_t(\nu_B) \\
	&\leq
	\frac{2^t s}{s - t} \theta(B)^2 c_B |B|^{s - t},
	\end{align*}
using Lemma~\ref{Itlemma} at the last step.
\end{proof}
\end{lemma}

\section{Probabilistic lemmas}
\begin{lemma} \label{deviationlemma}
Let $(S, \mathbf P)$ be a probability space and let $(A_n)_{n = 1}^\infty$
be a sequence of independent events in $S$. Let
	\[
	M_n = \sum_{k = 1}^n \mathbf P(A_k)
	\qquad\text{and}\qquad
	N_n(s) = \# \left\{ k \leq n; \, s \in A_k\right\},
	\]
and assume that\/ $\lim_{n \to \infty} M_n / \log n = \infty$.
Then $\mathbf P$-a.s.,
	\[
	\lim_{n \to \infty} \frac{N_n}{M_n} = 1.
	\]

\begin{proof}
Let $\varepsilon \in (0, 1]$. For every $s \in S$,
	\[
	N_n(s) = \sum_{k = 1}^n \chi_{A_k}(s),
	\]
and thus $\E N_n = M_n$. Thus, by the Chernoff
bounds~\cite[Equations~(6) and~(7)]{HagerupRub90}
for sums of independent random variables,
	\[
	\mathbf P (\{s; \, \left|N_n(s) - M_n \right| > \varepsilon M_n\})
	\leq
	2e^{-\frac{\varepsilon^2 M_n}{3}} =
	2n^{-\frac{\varepsilon^2}{3}\frac{M_n}{\log n}}.
	\]
This is summable in $n$, so by the Borel--Cantelli lemma there are
$\mathbf P$-a.s.~only finitely many $n$ such that
$\left|N_n - M_n \right| > \varepsilon M_n$. Letting $\varepsilon \to 0$ along
a countable set concludes the proof.
\end{proof}
\end{lemma}

The following lemma is a weaker variant of Lemma~\ref{problemma},
sufficient for the proof of Theorem~\ref{mainthm}.

\begin{lemma} \label{restrlemma}
Let $\mu$ be a probability measure on $\rea^d$, let $A$ be a Borel subset of\/ $\rea^d$ of
positive $\mu$-measure and let $\mu_A = \restr{\mu}{A} / \mu(A)$.
Then $f_{\mu_A}(\alpha + \varepsilon) \leq f_\mu(\alpha)$ for every $\varepsilon > 0$.

\begin{proof}
For $\omega \in \Omega$, let $(n_k(\omega))$ be the strictly increasing
enumeration of $\{n; \, \omega_n \in A\}$, and define the map
$\pi \colon \Omega \to \Omega$ by $\pi(\omega)_k = \omega_{n_k(\omega)}$.
By Lemma~\ref{deviationlemma} (with $S = \Omega$, $\mathbf P = \P{\mu}$
and $A_n = \{ \omega; \, \omega_n \in A\}$), $\P{\mu}$-a.e.~$\omega$ satisfies
	\[
	\lim_{k \to \infty} \frac{k}{n_k(\omega)} = \mu(A),
	\]
and for such $\omega$,
	\begin{align*}
	E_{\alpha + \varepsilon}(\pi(\omega))
	&=
	\limsup_k \cball{\omega_{n_k(\omega)}}{k^{-(\alpha + \varepsilon)}}
	\subset
	\limsup_k \cball{\omega_{n_k(\omega)}}{n_k^{-\alpha}} \\
	&=
	\limsup_{\omega_n \in A} \cball{\omega_n}{n^{-\alpha}}
	\subset E_\alpha(\omega).
	\end{align*}
Thus for $\P{\mu}$-a.e.~$\omega$,
	\[
	\dimh E_{\alpha + \varepsilon}(\pi(\omega)) \leq f_\mu(\alpha),
	\]
and the lemma follows since $\P{\mu} \circ \pi^{-1} = \P{\mu_A}$.
\end{proof}
\end{lemma}

\begin{lemma} \label{Wlemma}
  Let $\mu$ be a $(c, s)$-uniform probability measure on $\rea^d$, let
  $\varepsilon \in (0, s)$ and let $(n_k)_{k = 1}^\infty$ be a
  sequence of natural numbers such that
  \[
  \liminf_{k \to \infty} \log(n_k) / k > 0.
  \]
  Then for $\P{\mu}$-a.e.~$\omega \in \Omega$ there is a natural number
  $k_0 = k_0(\omega)$ such that for every $k \geq k_0$ and every ball
  $B$ in\/ $\rea^d$,
  \[
  \# \left(\{\omega_{n_k}, \ldots, \omega_{2n_k - 1} \} \cap B \right)
  \leq W \max\left(n_k r(B)^{s - \varepsilon}, 1 \right),
  \]
  where
  \[
  W = 4^s \max\left( ce, \, s \varepsilon^{-1}, \, 2 \right).
  \]

\begin{proof}
The set
	\[
	\Omega_0 = \left\{\omega; \,
	\mu\left(\cball{\omega_i}{r}\right) \leq cr^s
	\text{ for } i = 1, 2, \ldots, \text{ and } r > 0\right\}
	\]
has full $\P{\mu}$-measure. Let $E_n \subset \Omega_0$ be the event that there
is a ball $B$ such that
	\begin{equation} \label{badineq1}
	\# \left(\{\omega_n, \ldots, \omega_{2n - 1} \} \cap B \right) \geq
	W \max\left(n r(B)^{s - \varepsilon}, 1 \right) + 1.
	\end{equation}
It will be shown that $\sum_k \P{\mu}(E_{n_k}) < \infty$, so that the
statement follows by the Borel--Cantelli lemma.

Let $E_n^m \subset \Omega_0$ be the event that there is some $i \in \{n, \ldots, 2n - 1\}$
such that
	\begin{equation} \label{badineq2}
	\# \left(\{\omega_n, \ldots, \omega_{2n - 1} \} \cap \cball{\omega_i}{2^{-m}} \right) \geq
	4^{-s} W n 2^{-m(s - \varepsilon)} + 1.
	\end{equation}
If~\eqref{badineq1} holds for a ball $\cball{x}{r}$ such that $n r^{s - \varepsilon} < 1$
then~\eqref{badineq1} also holds for $\cball{x}{r_1}$, where $r_1$ is such that
$n r_1^{s - \varepsilon} = 1$, since replacing $r$ by $r_1$ increases the left side
and does not change the right side. Thus if $\omega \in E_n$ then~\eqref{badineq1} must
hold for some ball $B$ of radius $r$ such that $n r^{s - \varepsilon} \geq 1$,
and moreover~\eqref{badineq1} can only be satisfied
if $W r^{s - \varepsilon} < 1$, since otherwise the right side
is greater than $n$. It follows in particular from ~\eqref{badineq1} that there is
some $i \in \{n, \ldots, 2n - 1\}$ such that $\omega_i \in B$, and then~\eqref{badineq2}
holds with $m$ such that $2r \leq 2^{-m} < 4r$, since $B \subset \cball{\omega_i}{2^{-m}}$.
This shows that $E_n \subset \bigcup_{m = m_0}^{m_1(n)} E_n^m$, where
	\[
	m_0 = \floor{\frac{\log W}{(s - \varepsilon) \log 2}} - 2
	\qquad\text{and}\qquad
	m_1(n) = \floor{\frac{\log n}{(s - \varepsilon) \log 2}}.
	\]

Fix $n$ and $m \in [m_0, m_1(n)]$ and let $K = \ceil{4^{-s}Wn2^{-m(s - \varepsilon)}}$.
For $\omega$ to lie in $E_n^m$, there are $n$ choices for $i$ in~\eqref{badineq2}
and then there must be at least $K$ of the remaining $n - 1$ points in
$\{\omega_n, \ldots, \omega_{2n - 1} \}$ that lie in $\cball{\omega_i}{2^{-m}}$.
Since $\mu(\cball{\omega_i}{2^{-m}}) \leq c2^{-ms}$, this implies that
	\begin{align*}
	\P{\mu}(E_n^m) &\leq n \binom{n - 1}{K} \left(c2^{-ms}\right)^K
	\leq
	n \left(\frac{en}{K}\right)^K \left(c2^{-ms}\right)^K \\
	&\leq
	n2^{-\varepsilon m K}
	\leq
	n 2^{-4^{-s}W \varepsilon m n2^{-m(s - \varepsilon)}},
	\end{align*}
using a standard estimate for the binomial coefficient, that $W \geq 4^s c e$,
and that $m_0 \geq 1$. The last expression
has the form $n\exp(n\psi(m))$, where
\mbox{$\psi(x) = - \beta x e^{-\gamma x}$} for some positive $\beta$ and $\gamma$.
Since $\psi'(x) = \beta\left(\gamma x - 1\right) e^{-\gamma x}$,
the function $\psi$ is decreasing for $x < \gamma^{-1}$ and increasing for
$x > \gamma^{-1}$. Thus the maximum on the interval $[m_0, m_1(n)]$ of $\psi$,
and hence of $n\exp(n\psi(m))$, is attained at one of the endpoints.
This gives
	\begin{align*}
	\P{\mu}(E_n^m) &\leq
	n 2^{-4^{-s}W \varepsilon m_0 n2^{-m_0(s - \varepsilon)}}
	+
	n 2^{-4^{-s}W \varepsilon m_1(n) n2^{-m_1(n)(s - \varepsilon)}} \\
	&\leq
	n 2^{-4^{-s}W \varepsilon m_0 n2^{-m_0(s - \varepsilon)}}
	+
	2^{4^{-s}W\varepsilon} n^{1 - \frac{4^{-s}W\varepsilon}{s - \varepsilon}}.
	\end{align*}

Since $m_1(n)$ increases logarithmically in $n$, the estimate of $\P{\mu}(E_n^m)$
implies that
	\begin{align*}
	\P{\mu}(E_n)
	&\leq
	\text{const.} \times \left(
	n 2^{-4^{-s}W \varepsilon m_0 n2^{-m_0(s - \varepsilon)}} \log n+
	n^{1 - \frac{4^{-s}W\varepsilon}{s - \varepsilon}} \log n
	\right)
	\end{align*}
for $n \geq 2$. Here the first term is summable in $n$, using again that $m_0 \geq 1$,
and the second term is summable over $(n_k)$ since $W \geq 4^s s \varepsilon^{-1}$
and $(n_k)$ increases exponentially.
\end{proof}
\end{lemma}

\section{Proof of Theorem~\ref{mainthm}} \label{sec:proofmain}

If $\mathcal C$ is a finite subset of $\rea^d$, define the probability measure
	\[
	\nu_{\mathcal C} = \frac{1}{\#\mathcal C} \sum_{x \in \mathcal C} \delta_x.
	\]

\begin{lemma}	\label{collectlemma}
Let $\mu$ be a $(c, s)$-uniform probability measure on $\rea^d$, let $u \geq \udimp \mu$,
and let $\varepsilon \in (0, s)$. Let $(r_n)_{n = 1}^\infty$
be a sequence of positive numbers such that\/
$\lim_{n \to \infty} n r_n^{s - \varepsilon} = 0$, let $n_0$ be
a natural number, and let $B$ be a ball in $\rea^d$ such that $\mu(B) > 0$.
Let $W = 4^s \max\left( ce, \, s \varepsilon^{-1}, \, 2 \right)$. Then for
$\P{\mu}$-a.e.~$\omega$ there is a natural number $n \geq n_0$ and a set
$\mathcal C \subset \{\omega_n, \ldots, \omega_{2n - 1} \} \cap B$
such that
\begin{enumerate}[label=\roman*)]
\item \label{propi}
$\# \mathcal C \geq \frac{\mu(B) n}{2W}$,

\item \label{propii}
$\left|x_1 - x_2 \right| \geq 8 r_n$ for all pairs of distinct
points $x_1$ and $x_2$ in $\mathcal C$,

\item \label{propiii}
$\nu_{\mathcal C}$ is $(c_B, s- \varepsilon)'$-uniform, with
$c_B = \frac{2 W^2}{\mu(B)}$, and

\item \label{propiv}
$\mu\left( \cball{x}{r_n} \right) \geq r_n^{u + \varepsilon}$ for every
$x \in \mathcal C$.
\end{enumerate}

\begin{proof}
It may be assumed that $n_0$ is large enough so that
$n (8r_n)^{s - \varepsilon} \leq 1$ for every $n \geq n_0$.
Let 
	\[
	G_m = \{x \in B; \, \mu\left(\cball{x}{r_n}\right) \geq r_n^{u + \varepsilon}
	\text{ for all $n \geq m$}\}
	\]
and
	\[
	\mathcal C_n^\omega = \{\omega_i \in G_n; \, i \in \{n, \ldots, 2n - 1\}\}.
	\]
Since $(G_m)$ is increasing and by Lemma~\ref{deviationlemma},
	\begin{align*}
	\liminf_{n \to \infty} \frac{\# \mathcal C_n^\omega}{n}
	\geq
	\lim_{n \to \infty} \Bigg(
	&\frac{\# \mathcal \{\omega_i \in G_m; \, i \in \{1, \ldots, 2n - 1\}\}}{n} \\
	-
	&\frac{\# \mathcal \{\omega_i \in G_m; \, i \in \{1, \ldots, n - 1\}\}}{n}
	\Bigg)
	= \mu(G_m)
	\end{align*}
for every $m$ and $\P{\mu}$-a.e.~$\omega$. Moreover,
$\lim_{m \to \infty} \mu(G_m) = \mu(B)$ by the definition of $\udimp \mu$,
and thus
	\[
	\liminf_{n \to \infty} \frac{\# \mathcal C_n^\omega}{n} \geq \mu(B)
	\]
for $\P{\mu}$-a.e.~$\omega$. This together with Lemma~\ref{Wlemma} implies
that for $\P{\mu}$-a.e.~$\omega$ there is some $n \geq n_0$ such that
	\begin{align}
	&\# (\mathcal C_n^\omega \cap B) \geq \frac{\mu(B) n}{2}, \text{ and} \label{ketchup} \\
	&\#(\mathcal C_n^\omega \cap \cball{x}{r} ) \leq
	W \max\left(n r^{s - \varepsilon}, \, 1\right)
	\quad\text{for all } x \in \rea^d, \, r > 0. \label{majonais}
	\end{align}
Fix such $\omega$ and $n$, and for $i \in \{ n, \ldots, 2n - 1\}$ let
	\[
	M(i) = \left\{j \in \{ i + 1, \ldots, 2n - 1\}; \, |\omega_i - \omega_j| \leq
	8 r_n \right\};
	\]
then by~\eqref{majonais} and since $n (8r_n)^{s - \varepsilon} \leq 1$, the set $M(i)$
has at most $W - 1$ elements. Now $\mathcal C$ can be constructed recursively as follows.
For $j \in \{n, \ldots, 2n - 1\}$, assume that $\mathcal{C} \cap \{ \omega_n, \ldots, \omega_{j - 1} \}$
has be defined, and include $\omega_j$ in $\mathcal{C}$ if and only if $\omega_j \in B$ and
there is no $i < j$ such that $\omega_i \in \mathcal C$ and $j \in M(i)$.
Then it is clear that
$\mathcal C \subset \{\omega_n, \ldots, \omega_{2n - 1} \} \cap B$ and that
\itemref{propii} and~\itemref{propiv} hold.
If $\omega_j \in \mathcal C_n^\omega \setminus \mathcal C$
then there is some $i$ such that $\omega_i \in \mathcal C$ and $j \in M(i)$,
and thus
	\[
	\# \mathcal C_n^\omega \setminus \mathcal C
	\leq
	\sum_{i; \, \omega_i \in \mathcal C} \#M(i)
	\leq
	(W - 1) \# \mathcal C,
	\]
so that by~\eqref{ketchup},
	\[
	\# \mathcal C \geq \frac{\# \mathcal C_n^\omega}{W} \geq \frac{\mu(B)n}{2W},
	\]
that is, property~\itemref{propi} holds. Then property \itemref{propiii}
follows from \itemref{propi} and~\eqref{majonais}.
\end{proof}
\end{lemma}

\begin{lemma} \label{almosttherelemma}
Let $\mu$ be a $(c, s)$-uniform probability measure on $\rea^d$ and let
$u \geq \udimp \mu$. Then
	\[
	f_\mu(\alpha) \geq \frac{1}{\alpha} + s - u.
	\]

\begin{proof}
Let $\varepsilon \in (0, s)$ and
	\[
	t < \frac{1}{\alpha} + s - u - 3\varepsilon,
	\]
and for $n = 1, 2, \ldots$, let $r_n = (2n)^{-\alpha} / 2$.
A triple $(\mathcal B_0, R_0, \pi_0)$ is a \emph{pre-tree} if it satisfies the
conditions in the definition of a fractal tree, but with \itemref{treepropiii} replaced
by
\begin{enumerate}[label=\emph{iii')}]
\item
for every $B \in \mathcal B_0$, the balls in $\pi_0^{-1}(B)$ have the same
radius, are separated by a distance greater than twice that radius, and are centred
in $B$; and $2 \leq \#\pi_0^{-1}(B) < \infty$.
\end{enumerate}

For $\P{\mu}$-a.e.~$\omega$, the following recursive construction of a pre-tree
$(\mathcal B_0, R_0, \pi_0)$ can be done, since each step in the construction works
$\P{\mu}$-a.s.~and there are only countable many steps. Let $R_0$ be any closed
ball such that $\mu(R_0) > 0$ and put $R_0$ in $\mathcal B_0$. Assuming that
generation $g$ of the pre-tree has been defined and that every ball in $\mathcal B_0$
so far has positive $\mu$-measure, consider a ball $B$ in $\mathcal B_0$ of
generation $g$. Let $W = 4^s \max\left( ce, \, s \varepsilon^{-1}, \, 2 \right)$
and let $n_0$ be an integer large enough so that
	\[
	\frac{\mu(B) n}{2 W^2} \geq r_n^{-1/\alpha + \varepsilon}
	\]
for $n \geq n_0$. Then for
$\P{\mu}$-a.e.~$\omega$ there is some $n \geq n_0$ and a set
$\mathcal C \subset \{\omega_n, \ldots, \omega_{2n - 1}\} \cap B \cap \supp\mu$
having the properties \emph{i)--iv)} of Lemma~\ref{collectlemma}.
Extend $(\mathcal B_0, R_0, \pi_0)$ by setting
	\[
	\pi_0^{-1}(B) = \left\{ \cball{x}{r_n}; \, x \in \mathcal C \right\}.
	\]
Then every ball in $\pi_0^{-1}(B)$ has positive $\mu$-measure, and
	\begin{equation} \label{nbrCesteq}
	\# \mathcal C \geq r_n^{-1/\alpha + \varepsilon}
	\end{equation}
by the choice
of $n_0$ and property~\emph{i)} of Lemma~\ref{collectlemma}.

For $\omega$ such that $(\mathcal B_0, R_0, \pi_0)$ can be constructed, let
$R = \cball{x(R_0)}{2r(R_0)}$ and
	\[
	\mathcal B = \left\{\cball{x}{2r}; \, \cball{x}{r} \in \mathcal B_0 \right\},
	\]
and define $\pi$ in the obvious way. Then
$(\mathcal B, R, \pi)$ is a fractal tree and $K(\mathcal B) \subset E_\alpha(\omega)$
since every ball in $\mathcal B \setminus \{R\}$ is a subset of a ball of the form
$\cball{\omega_k}{k^{-\alpha}}$. By property~\emph{\ref{propii}}
of Lemma~\ref{collectlemma}, if $B_1$ and $B_2$ are distinct balls in $\mathcal B$
such that $\pi(B_1) = \pi(B_2)$, then
	\[
	\dist\left( B_1, \, B_2 \right) \geq \frac{1}{2} \left| x(B_1) - x(B_2)\right|.
	\]
Define the probability measure
$\theta$ on $K(\mathcal B)$ as in Section~\ref{treesec}.
For $A \in \mathcal B$, let $\rho(A)$ denote the radius of the balls in
$\pi^{-1}(A)$. If $a$ and $b$ are two elements of
$[0, \infty]$, let $a \lesssim b$ mean that $a$ is finite or $b$ is infinite.
Thus $a \lesssim b$ if and only if there is a positive
number $\gamma$ such that $a \leq \gamma b$ or $a \leq b + \gamma$.
Then,
	\begin{align*}
	I_t(\theta)
	&\lesssim
	\sum_{B \in \mathcal B} \theta(B)^2 c_B |B|^{s - \varepsilon - t}
	\lesssim
	\sum_{A \in \mathcal B} \Bigg(\rho(A)^{s - \varepsilon - t}
	\sum_{B \in \pi^{-1}(A)} \theta(B)^2 c_B \Bigg)\\
	&\lesssim
	\sum_{A \in \mathcal B} \Bigg(
	\frac{\theta(A)\rho(A)^{s - \varepsilon - t}}{\left(\# \pi^{-1}(A)\right)^2}
	\sum_{B \in \pi^{-1}(A)} \frac{1}{\mu(B)} \Bigg)
	\lesssim
	\sum_{A \in \mathcal B}
	\theta(A)\rho(A)^{s + 1 / \alpha - t - u - 3\varepsilon}.
	\end{align*}
Here the first step is by Lemma~\ref{treelemma}, the second is valid
since every ball $B \in \mathcal B \setminus \{R\}$ appears in
the sum on the right and $|B| = 2\rho(A)$ for $B$ in the inner sum,
the third step uses that $\theta(B) = \theta(A) / \#\pi^{-1}(A)$
and property~\emph{\ref{propiii}} of Lemma~\ref{collectlemma},
and the last step holds since
$\#\pi^{-1}(A) \geq \rho(A)^{-1/\alpha + \varepsilon}$ by~\eqref{nbrCesteq},
and by property~\emph{\ref{propiv}} of Lemma~\ref{collectlemma}.
Now, if $A$ is of generation $g$ then $\rho(A) \leq D2^{-g}$
where $D$ is the diameter of the balls of generation $1$, and
the sum of $\theta(A)$ when $A$ runs through a generation
is $1$. Thus
	\[
	I_t(\theta)
	\lesssim
	\sum_{g = 0}^\infty \Bigg(
	2^{-g(s + 1 / \alpha - t - u - 3\varepsilon)}
	\sum_{\substack{A \in \mathcal B \\ g(A) = g}} \theta(A)\Bigg)
	=
	\sum_{g = 0}^\infty
	2^{-g(s + 1 / \alpha - t - u - 3\varepsilon)}
	< \infty,
	\]
by the choice of $t$.
Thus $\dimh E_\alpha \geq t$ almost surely, and the lemma follows since
$t$ can be taken arbitrarily close to $1 / \alpha + s - u$.
\end{proof}
\end{lemma}

\begin{proof}[Proof of Theorem~\ref{mainthm}]
By the definition of $\delta$ there is an $\varepsilon_0$ such that
for every $\varepsilon \in (0, \varepsilon_0)$, there is a
Borel set $A_1$ with positive $\mu$-measure such that
	\[
	\ldim{\mu}{x} > \frac{1}{\alpha} + \varepsilon
	\qquad\text{and}\qquad
	\delta_\mu(x) \leq \delta + \varepsilon
	\]
for all $x \in A_1$. Fix $\varepsilon$ and $A_1$.
Then there there is some $s > 1 / \alpha$ and a Borel set $A_2 \subset A_1$ with positive
$\mu$-measure such that $\ldim{\mu}{x} \in [s + \varepsilon, s + 2\varepsilon)$ for $x \in A_2$.
Thus by the definition of $\delta_\mu$,
	\[
	s + \varepsilon \leq \ldim{\mu}{x} \leq \udim{\mu}{x} \leq u
	\]
for $x \in A_2$, where $u = s + \delta + 3\varepsilon$. Next, there is
a positive number $c$ and a Borel set $A_3 \subset A_2$ with positive
$\mu$-measure, such that whenever $x \in A_3$ then
\[
\mu(\cball{x}{r}) \leq cr^s
\qquad\text{for all }
r > 0.
\]
Let $\mu_1 = \restr{\mu}{A_3} / \mu (A_3)$.

Consider an arbitrary Borel subset $A$ of $\rea^d$. If $A \cap A_3 =
\emptyset$ then $\mu_1 (A) = 0$. Otherwise there is some $x \in A \cap
A_3$, and
\[
\mu_1 (A) \leq \mu_1 (\cball{x}{|A|}) \leq
\frac{\mu(\cball{x}{|A|})}{\mu(A_3)} \leq \frac{c}{\mu(A_3)} |A|^s.
\]
This shows that $\mu_1$ is $(c_1, s)$-uniform, with
$c_1 = 2^s c / \mu(A_3)$.

By the Lebesgue--Besicovitch differentiation theorem
\cite[Section~1.7.1]{EvansGariepy}, there is a Borel set $A_4 \subset
A_3$ such that $\mu(A_4) = \mu(A_3)$ and
\[
\lim_{r \to 0} \frac{\mu(\cball{x}{r} \cap A_3)}{\mu(\cball{x}{r})} =
1
\]
for all $x \in A_4$. Thus for $\mu_1$-a.e.~$x$,
	\begin{align*}
	\udim{\mu_1}{x} &=
	\limsup_{r \to 0} \frac{\log \mu(\cball{x}{r} \cap A_3) - \log \mu(A_3)}{\log r} \\
	&=
	\limsup_{r \to 0} \frac{\log \mu(\cball{x}{r})}{\log r} = \udim{\mu}{x} \leq u.
	\end{align*}
By Lemma~\ref{restrlemma} and Lemma~\ref{almosttherelemma},
	\[
	f_\mu(\alpha) \geq
	f_{\mu_1} (\alpha + \varepsilon) \geq
	\frac{1}{\alpha + \varepsilon} + s - u =
	\frac{1}{\alpha + \varepsilon} - \delta - 3\varepsilon.
	\]
Letting $\varepsilon \to 0$ concludes the proof.
\end{proof}

\section{Proof of Proposition~\ref{pro:main}} \label{sec:proofprop}

\subsection{The lower bound in Proposition~\ref{pro:main}}

\begin{lemma} \label{lem:souslin}
  Let $\mu$ be a probability measure on $\rea^d$ and let $A$ be a
  Souslin set such that
  \[
  A \subset \{x; \, x \in E_\alpha (\omega) \text{ for }
  \P{\mu}\text{-a.e.}~\omega\}.
  \]
  Then
  \[
  f_\mu(\alpha) \geq \dimh A.
  \]
\end{lemma}

\begin{proof}
  Let $\varepsilon > 0$. Since $A$ is a Souslin set, Frostman's lemma
  implies that there is a probability measure $\nu$ such that $\nu(A)
  = 1$ and $\nu$ is $(c, s)$-uniform with $s = \dimh A - \varepsilon$.
In particular, $\udimh \nu \geq \dimh A - \varepsilon$. The expected
  value of $\nu(E_\alpha (\omega))$ with respect to $\P{\mu}$ is given by
  \begin{align*}
    \E\left(\nu(E(\omega))\right) &= \iint \chi_{E_\alpha (\omega)}(x) \,
    \mathrm{d}\nu(x) \, \mathrm{d}{\P{\mu}}(\omega) \\ &= \iint
    \chi_{E_\alpha (\omega)}(x) \, \mathrm{d}{\P{\mu}}(\omega) \,
    \mathrm{d}\nu(x) \\ &= \int \P{\mu}\left(\{ \omega \in \Omega; \, x
    \in E_\alpha (\omega) \}\right) \, \mathrm{d}\nu(x) = 1,
  \end{align*}
  and hence $\nu(E_\alpha (\omega)) = 1$ for
  $\P{\mu}$-a.e.~$\omega$. It follows that $\P{\mu}$-a.s.,
  \[
  \dimh E_\alpha (\omega) \geq \udimh \nu
%	\geq \ldimh \nu
	\geq \dimh A - \varepsilon.
  \]
  Letting $\varepsilon \to 0$ along a countable set concludes the proof.
\end{proof}

As an application of Lemma~\ref{lem:souslin}, we can now prove the
lower bound of $f_\mu (\alpha)$ in Proposition~\ref{pro:main}.

\begin{corollary} \label{cor:lowerbound}
For every $\alpha > 0$,
	\[
	\lim_{s \to 1/\alpha^-} F_\mu(s)
	\leq f_\mu(\alpha).
	\]
\end{corollary}

\begin{proof}
  By Lemma~\ref{lem:souslin} it suffices to show that if
  $\ldim{\mu}{x} < 1 / \alpha$, then $x \in E_\alpha (\omega)$
  $\P{\mu}$-a.s. So assume that $\ldim{\mu}{x} < 1 / \alpha$. Then there
  is a positive constant $C$ and a sequence $(r_k)_{k = 1}^\infty$
  converging to $0$ such that $r_{k + 1}^{1 / \alpha}
  \ceil{r_k^{-1/\alpha}} \leq 1 / 2$ and
  \[
  \mu(B(x, r_k)) \geq C r_k^{1 / \alpha}
  \]
  for every $k$.

  If $n \leq r_{k + 1}^{-1/\alpha}$ then
  \[
  \P{\mu} \left(x \in B (\omega_n, n^{-\alpha})\right) = \mu(B(x, n^{-\alpha}))
  \geq \mu(B(x, r_{k + 1})) \geq C r_{k + 1}^{1 / \alpha},
  \]
  and thus
  \[
  \sum_{n = 1}^\infty \P{\mu} \left(x \in B (\omega_n, n^{-\alpha})\right) \geq C
  \sum_{k = 1}^\infty r_{k + 1}^{1 / \alpha} \left( \floor{r_{k +
      1}^{-1/\alpha}} - \ceil{r_k^{-1/\alpha}} \right).
  \]
  For all large enough $k$, the $k$:th term in the sum is greater than
  or equal to $1 / 3$, and hence the sum diverges. It then follows by
  the Borel--Cantelli lemma that $\P{\mu}$-a.s., $x \in E_\alpha (\omega)$.
\end{proof}

\subsection{The upper bound in Proposition~\ref{pro:main}}

\begin{proposition}	\label{pro:ubound1}
For every $\alpha > 0$,
	\[
	f_\mu (\alpha) \leq \max\left(F_\mu(1/\alpha), \,
	1 / \alpha + \sup_{s \geq 1 / \alpha} (G_\mu(s) - s) \right).
	\]
\end{proposition}

\begin{proof}
  We will use the same method to prove the statement as was used by
  Seuret \cite[Proposition~5]{Seuret}.

Let
  \[
  \mathscr{D}_n = \{ [k_1 2^{-n}, (k_1+1) 2^{-n}) \times \cdots
    \times [k_d 2^{-n}, (k_d+1) 2^{-n} ) ; \, k_j \in \ints \},
  \]
and given a point $x$ and an integer $n$, let $D_n (x)$ be the
  unique $D \in \mathscr{D}_n$ such that $x \in D$.
Let $s > 0$ and $\varepsilon \in (0, \varepsilon_0)$, where $\varepsilon_0$
is a positive number that will be specified. Let
  \begin{align*}
    A_\varepsilon(n,s) &=
	\left\{ x ; \, 2^{-(s+\varepsilon)n} \leq \mu(D_n (x))
    \leq 2^{-(s-\varepsilon)n} \right\}, \\
	A_\varepsilon(s) &= \limsup_{n\to\infty} A_\varepsilon(n,s).
  \end{align*}
If $x \in A_\varepsilon(s)$ then
	\[
	\mu \left(\oball{x}{\sqrt{d}2^{-n}}\right)
	\geq \mu(D_n(x)) \geq 2^{-(s+\varepsilon) n}
	\]
for infinitely many $n$, so that $\ldim{\mu}{x} \leq s+\varepsilon$. Thus
  \[
  A_\varepsilon(s) \subset \left\{ x ; \, \ldim{\mu}{x} \leq s + \varepsilon \right\}.
  \]

Recall that the coarse spectrum is defined as
  \[
  G_\mu (s) = \lim_{\varepsilon \to 0} \limsup_{r \to 0}
  \frac{\log (N_r (s + \varepsilon) - N_r (s - \varepsilon))}{- \log
    r},
  \]
  where $N_r (s)$ denotes the number of $d$-dimensional cubes of
  the form
  \[
  Q = [k_1 r, (k_1+1) r) \times \ldots \times [k_d r, (k_d+1) r)
  \]
  with $k_1, \ldots, k_d \in \mathbf{Z}$ and $\mu (Q) \geq r^s$.
Let $M_\varepsilon (n, s)$ be the number of dyadic cubes in
$\mathscr{D}_n$ that are included in $A_\varepsilon(n, s)$.
Then
  \[
  G_\mu (s) = \lim_{\varepsilon \to 0} \limsup_{n \to \infty}
  \frac{\log M_\varepsilon (n,s)}{ n \log 2},
  \]
so given $\gamma > 0$, it is possible to choose $\varepsilon_0$
such that
  \[
  M_\varepsilon (n,s) \leq 2^{(G_\mu (s) + \gamma)n}
  \]
for all $\varepsilon \in (0, \varepsilon_0)$ and all
$n \geq n_0(\varepsilon)$.

Let
  \[
  K(n) = \left\{ k ; \, 2^{n/\alpha} \leq k < 2^{(n+1)/\alpha} \right\}
  \]
and let $\tilde{\mathscr{D}}_n$ denote the set of dyadic cubes
$D$ in $\mathscr{D}_n$ such that $\omega_k \in D$ for some $k \in K(n)$
and $D \subset A_\varepsilon(n, s)$. Then for every $m$,
	\begin{equation} \label{covereq}
	E_\alpha(\omega) \cap A_\varepsilon(s) \subset
	\bigcup_{n \geq m} \bigcup_{D \in \tilde{\mathscr{D}}_n} \hat D,
	\end{equation}
where $\hat D$ is the cube that is concentric with $D$ and has $3$ times
the sidelength of $D$, using that if $k \in K(n)$ and
$\omega_k \in D \in \tilde{\mathscr{D}}_n$ then
$\cball{\omega_k}{k^{-\alpha}} \subset \hat D$. Since
  \[
  \mu\left(A_\varepsilon(n,s)\right) \leq M_\varepsilon (n,s) 2^{-(s-\varepsilon)n},
  \]
the expected number of dyadic cubes in $\tilde{\mathscr{D}}_n$ satisfies
  \begin{align*}
    \E\big(\# \tilde{\mathscr{D}}_n\big) \leq \mu (A_\varepsilon(n,s))
    \# K(n) & \leq M_\varepsilon (n,s) 2^{-(s-\varepsilon)n}
    2^{(n+1)/\alpha} \\ & \leq 2^{1/\alpha} 2^{(1/\alpha - s + G_\mu
      (s) + \varepsilon + \gamma)n}.
  \end{align*}
Thus by Markov's inequality,
  \[
  \P{\mu} \left(\# \tilde{\mathscr{D}}_n \geq 2^{(1/\alpha - s +
    G_\mu (s) + 2 \varepsilon + \gamma)n}\right) \leq \frac{\E \big(\#
    \tilde{\mathscr{D}}_n\big)}{ 2^{(1/\alpha - s + G_\mu (s) +
      2 \varepsilon + \gamma)n}} \leq 2^{1/\alpha} 2^{-\varepsilon n}.
  \]
  It follows by the Borel--Cantelli lemma that almost surely, there is
  an $N$ such that for each $n > N$, the number of dyadic cubes in
  $\tilde{\mathscr{D}}_n$ satisfies
  \begin{equation} \label{eq:numberofcubes}
  \# \tilde{\mathscr{D}}_n \leq 2^{(1/\alpha - s + G_\mu
    (s) + 2 \varepsilon + \gamma)n}.
  \end{equation}
  
Assume that there is such an $N$. 
Let $\delta > 0$ and choose $m \geq N$ such that
$3 \sqrt{d} 2^{-m} \leq \delta$. Then the cubes used to cover
$E_\alpha (\omega) \cap A_\varepsilon(s)$ in~\eqref{covereq} have diameters
less than $\delta$, and thus
  \begin{align*}
    \hmeas_\delta^t\left(E_\alpha (\omega) \cap
    A_\varepsilon(s)\right) \leq \sum_{n \geq m} \sum_{D \in
      \tilde{\mathscr{D}}_n} |\hat D |^t &\leq \sum_{n \geq m} 2^{(1/\alpha
      - s + G_\mu (s) + 2 \varepsilon + \gamma)n} \cdot 3^t d^{\,t / 2}
    2^{-nt} \\ &= 3^t d^{\,t / 2} \sum_{n \geq m} 2^{(1/\alpha - s + G_\mu
      (s) - t + 2 \varepsilon + \gamma)n}.
  \end{align*}
The series converges for $t > 1/\alpha - s + G_\mu (s) + 2 \varepsilon + \gamma$,
and it follows that
	\begin{equation} \label{eq:dimEAs}
	\dimh \left(E_\alpha (\omega) \cap A_\varepsilon(s)\right)
	\leq
	1 / \alpha + (G_\mu (s) - s) + 2 \varepsilon + \gamma.
	\end{equation}
Since almost surely there is an $N$ such that \eqref{eq:numberofcubes} holds,
\eqref{eq:dimEAs} holds almost surely for fixed $s$, $\gamma$ and $\varepsilon$.

  Let $s_j = 1/\alpha + j \varepsilon$ and write
  \[
  E_\alpha (\omega) = (E_\alpha (\omega) \cap \{ x ; \, \ldim{\mu}{x}
  \leq 1/\alpha \}) \cup \bigcup_{j = 1}^\infty (E_\alpha (\omega)
  \cap A_\varepsilon(s_j)).
  \]
Then by~\eqref{eq:dimEAs},
	\begin{align*}
	\dimh E_\alpha (\omega) &\leq \max \Bigl\{ F_\mu(1 / \alpha),
        \, \max_j \dimh (E_\alpha (\omega) \cap A_\varepsilon(s_j))
        \Bigr\} \\ &\leq \max\Bigl\{ F_\mu(1 / \alpha), \, 1 / \alpha +
        \max_j (G_\mu (s_j) - s_j) + 2 \varepsilon + \gamma \Bigr\},
	\end{align*}
and letting $\varepsilon \to 0$ and then $\gamma \to 0$ concludes the
proof.
\end{proof}

\section{Proof of Proposition~\ref{pro:alternative}} \label{sec:alternative}

If $K$ is a subset of $\rea^d$ we let $K_r$ denote the $r$-fattening of $K$,
that is,
\[
K_r = \{x; \, \dist(x, K) \leq r \}.
\]

\begin{lemma} \label{udimblemma}
  Let $\beta$, $c$, $s$ and $\varepsilon$ be positive numbers, let $K$
  be a subset of $\rea^d$ such that\/ $\udimb K = \beta$ and let $\mu$
  be a measure on $\rea^d$ such that for every $x \in K$ and
  every $r > 0$, it holds that $\mu\left(\cball{x}{r}\right) \leq c
  r^s$.  Then there is a constant $C$ such that for every $r > 0$,
  \[
  \mu(K_r) \leq C r^{s - (\beta + \varepsilon)}.
  \]

\begin{proof}
  For $r > 0$, let $N(r)$ be the minimal number of closed balls of
  radius $r$ needed to cover $K$. By the definition of the upper box
  dimension, there is then a constant $\gamma$ such that $N(r) \leq
  \gamma r^{-(\beta + \varepsilon)}$ for every $r > 0$.

  Fix $r > 0$ and let $y_1, \ldots, y_{N(r / 2)}$ be points in
  $\rea^d$ such that the balls $\left\{\cball{y_i}{r / 2}\right\}_{i =
    1}^{N(r / 2)}$ cover $K$. Then every $\cball{y_i}{r / 2}$
  intersects $K$, since otherwise $K$ could be covered by $N(r / 2) -
  1$ balls of radius $r / 2$, contradicting the minimality of $N(r /
  2)$. Thus there are $x_1, \ldots, x_{N(r / 2)} \in K$ such that
  \[
  K \subset \bigcup_{i = 1}^{N(r / 2)} \cball{x_i}{r},
  \qquad\text{and hence}\qquad
  K_r \subset \bigcup_{i = 1}^{N(r / 2)} \cball{x_i}{2r}.
  \]
  It follows that
  \[
  \mu(K_r) \leq N(r / 2) \cdot c (2r)^s \leq
  2^{s + \beta + \varepsilon} c\gamma r^{s - (\beta + \varepsilon)}.
  \qedhere
  \]
\end{proof}
\end{lemma}

We can now prove Proposition~\ref{pro:alternative}. The proof is based
on Lemma~\ref{udimblemma}, but is otherwise very similar to the proof
of Proposition~\ref{pro:ubound1}.

\begin{proof}[Proof of Proposition~\ref{pro:alternative}]
For $a \leq b$ let
	\[
	A_a^b = \{x \in \rea^d; \, a \leq \ldim{\mu}{x} \leq b\}.
	\]
Let $\varepsilon > 0$ and $s_j = 1 / \alpha + j\varepsilon$ for $j = 0, 1, \ldots$
Then,
	\[
	E_\alpha(\omega) = \Big(E_\alpha(\omega) \cap A_0^{s_0}\Big) \cup
	\Big(\bigcup_{j = 0}^\infty
	E_\alpha(\omega) \cap A_{s_j}^{s_{j + 1}}\Big),
	\]
and thus
	\begin{equation}	\label{maxineq}
	\dimh E_\alpha(\omega) \leq
	\max\left(H_\mu(1 / \alpha), \,
	\max_{j \geq 0} \dimh E_\alpha(\omega) \cap A_{s_j}^{s_{j + 1}}
	\right).
	\end{equation}
It will be shown for $s \geq 1 / \alpha$ that for $\P{\mu}$-a.e.~$\omega$,
	\begin{equation} \label{aimineq}
	\dimh \left(E_\alpha(\omega) \cap A_{s}^{s + \varepsilon}\right)
	\leq H_\mu(s + \varepsilon) + \frac{1}{\alpha} - s + 4
	\varepsilon;
	\end{equation}
letting $\varepsilon \to 0$ in~\eqref{maxineq} then proves the proposition,
using that $\bar H_\mu = \tilde H_\mu$ since $H_\mu$ is increasing.
  
  Fix $s \geq 1 / \alpha$ and $\varepsilon > 0$ and let $\beta =
  H_\mu (s + \varepsilon) + \varepsilon$.  The set $A_s^{s +
    \varepsilon}$ can be decomposed as
  \[
  A_s^{s + \varepsilon} = \bigcup_{m = 1}^\infty K_m,	
  \]
  where
  \[
  K_m = \left\{x \in A_s^{s + \varepsilon}; \, \mu(\cball{x}{r})
  \leq mr^{s - \varepsilon}
  \text{ for all } r > 0\right\},
  \]
  and each $K_m$ can be decomposed as
  \[
  K_m = \bigcup_{i = 1}^\infty K_m^i,
  \]
  where $\udimb K_m^i \leq \beta$ for all $i$ (using that $K_m
  \subset A_0^{s + \varepsilon}$ and~\eqref{dimpeq}).

  Fix $m$ and $i$, and for $\omega \in \Omega$ let
  \[
  G(\omega) =
  \left\{k ; \, \cball{\omega_k}{k^{-\alpha}} \cap K_m^i \neq
  \emptyset\right\}
  \]
  and let $\{k_n(\omega)\}_{n = 1}^\infty$ be the strictly increasing
  enumeration of $G(\omega)$. By Lemma~\ref{udimblemma}, there is a
  constant $C_1$ such that for every $r > 0$,
  \[
  \mu\left((K_m^i)_r\right) \leq C_1 r^{s - \beta - 2\varepsilon},
  \]
  so that
  \[
  \E\big(\# G(\omega) \cap [1, 2^n] \big)
  \leq C_2 2^{(1 - \alpha(s - \beta - 2\varepsilon))n}
  \]
  for some constant $C_2$ and all $n$. Thus by Markov's inequality,
  \[
  \P{\mu}\left(\left\{\omega; \, \# G(\omega) \cap [1, 2^n] \geq 2^{(1 -
    \alpha(s - \beta - 3\varepsilon))n}\right\}\right) \leq C_2
  2^{-\varepsilon n}.
  \]
  The Borel--Cantelli lemma then implies that for
  $\P{\mu}$-a.e.~$\omega$ there is a constant $C_3 = C_3(\omega)$
  such that
  \[
  \# G(\omega) \cap [1, 2^n] \leq C_3 2^{(1 - \alpha(s - \beta -
    3\varepsilon))n}
  \]
  for all $n$. Thus for a.e.~$\omega$ there is a constant $C_4 =
  C_4(\omega)$ such that
  \[
  k_n(\omega) \geq C_4 n^{1 / (1 - \alpha(s - \beta - 3\varepsilon))},
  \]
  or equivalently,
  \[
  k_n(\omega)^{-\alpha} \leq C_4^{-\alpha} n^{-\alpha / (1 -
    \alpha(s - \beta - 3\varepsilon))}
  \]
  for all $n$. For such $\omega$, the trivial upper bound for
  the Hausdorff dimension gives
  \begin{align*}
    \dimh \left(E_\alpha(\omega) \cap K_m^i \right)&= \dimh
    \left(\limsup_n \cball{x}{k_n(\omega)^{-\alpha}}\right)\\ &\leq
    \frac{1 - \alpha(s - \beta - 3\varepsilon)}{\alpha} =
    \frac{1}{\alpha} - s + H_\mu (s + \varepsilon) + 4\varepsilon.
  \end{align*}
  The inequality~\eqref{aimineq} follows by the countable stability of
  the Hausdorff dimension.
\end{proof}

\section{Proofs to Example~\ref{example}} \label{sec:example}

\subsection{Some lemmas}
In this section the analogue of $f_\mu$ in a more general setting will
be considered. Let $X$ be a separable metric space (separable in order
to ensure that measures have well-defined supports), let $\seq r =
(r_k)_{k = 1}^\infty$ be a sequence of positive numbers and let $\seq
\mu = (\mu_k)_{k = 1}^\infty$ be a sequence of probability measures on
$X$. Let $\Omega = X^\nat$ and let $\P{\seq \mu} = \times_{k =
  1}^\infty \mu_k$.  For $\omega \in \Omega$, let
\[
E_{\seq r}(\omega) = \limsup_k \cball{\omega_k}{r_k}.
\]
As before, $\dimh E_{\seq r}(\omega)$ is
$\P{\seq\mu}$-a.s.~constant---denote the a.s.~value by
$f_{\seq \mu}(\seq r)$. If $\seq \mu = (\mu)$ is a constant sequence,
then $f_{\seq \mu}$ will also be denoted by $f_\mu$, and if
$\seq r = (k^{-\alpha})_{k = 1}^\infty$ then $f_{\seq \mu}(\seq r)$
will also be denoted by $f_{\seq \mu}(\alpha)$.

Let $\Sigma = \{0, 1\}^\nat$, and if $p$ is a number in $[0, 1]$, let
$\P{p}$ be the probability measure on $\Sigma$ given by
$\P{p} = \times_{k = 1}^\infty \left((1 - p) \delta_0 + p \delta_1\right)$.

In the proofs of Lemma~\ref{posproblemma} and \ref{convcomblemma} below,
the only properties of the Hausdorff dimension that are used are
that it is monotone,
  \[
  A \subset B \quad \Rightarrow \quad \dimh A \leq \dimh B,
  \]
and that it is finitely stable,
  \[
  \dimh \bigcup_{k=1}^n A_k = \max \left\{ \dimh A_k ; \, 1 \leq k \leq n \right\}.
  \]
Hence the statements of these lemmas, and also Lemma~\ref{problemma}, are
still true if the Hausdorff dimension is replaced by any other ``dimension''
that has these properties, for instance the packing dimension.

\begin{lemma} \label{posproblemma}
  Let $(A_k)_{k = 1}^\infty$ be a sequence of subsets of $X$ and for
  $\sigma \in \Sigma$, let
  \[
  A(\sigma) = \limsup_{\sigma_k = 1} A_k.
  \]
  Let $p \in (0, 1]$. Then for $\P{p}$-a.e.~$\sigma$,
  \[
  \dimh A(\sigma) = \dimh A,
  \]
  where $A = \limsup_k A_k$.

  \begin{proof}
    Let $H = (\rea / \ints)^\nat$ and let $\Q$ be the product measure
    on $H$ that projects to Lebesgue measure on each component. Define
    the map $\pi \colon H \to \Sigma$ by
    \[
    \pi(\eta)_k =
    \begin{cases}
      1 & \text{if } \eta_k \in [0, p) \\
	0 & \text{otherwise}.
    \end{cases}
    \]
    Then $\P{p} = \Q \circ \pi^{-1}$, and thus it suffices to show that
    $\dimh A(\pi(\eta)) = \dimh A$ for $\Q$-a.e.~$\eta$. The
    inequality $\leq$ holds for \emph{every} $\eta$ since
    $A(\pi(\eta)) \subset A$.

    To see the opposite inequality, let $N$ be a natural number such
    that $Np \geq 1$. Fix $\eta$ and a point $x \in X$, and for $i =
    0, \ldots, N - 1$, let
    \[
    K_i = \{ k; \, x \in A_k \text{ and } \eta_k \in [0, p) + i / N\}.
    \]
    Then $x$ lies in $A$ if any only if $\bigcup_{i = 0}^{N - 1} K_i$
    is infinite, and this happens if and only if there is some $i$
    such that $K_i$ is infinite. Thus
    \[
    A = \bigcup_{i = 0}^{N - 1}
    \limsup \{A_k; \, \eta_k \in [0, p) + i / N \}
      = \bigcup_{i = 0}^{N - 1}
      A\big(\pi(\eta - (i / N)^\infty)\big),
    \]
    where the addition in the last expression is coordinate-wise in
    $\rea / \ints$ and $(i / N)^\infty$ denotes the constant sequence
    with all entries equal to $i / N$.  Hence by the finite stability
    of the Hausdorff dimension, there is for every $\eta$ some $i$
    such that
    \[
    \dimh A\left(\pi(\eta - (i / N)^\infty)\right) = \dimh A,
    \]
    and it follows that
    \[
    \sum_{i = 0}^N \chi_G(\eta - (i / N)^\infty) \geq 1,
    \]
    where
    \[
    G = \left\{ \eta \in H; \, \dimh A(\pi(\eta)) \geq \dimh A \right\}.
    \]
    Integrating this, and using the translation invariance of $\Q$,
    gives
    \[
    1 \leq \sum_{i = 0}^N \int \chi_G(\eta - (i / N)^\infty) \,
    \intd\Q (\eta) = N \int \chi_G(\eta) \, \intd\Q(\eta) = N \Q(G).
    \]
    This shows that the event $\dimh A(\pi(\eta)) \geq \dimh A$ has
    positive $\Q$-proba\-bility, and since it is a tail event it must
    then have full probability.
  \end{proof}
\end{lemma}

\begin{lemma} \label{convcomblemma}
  Let $\seq \mu^0$ and $\seq \mu^1$ be two sequences of probability
  measures on $X$, let $p \in (0, 1)$, and define the sequence $\seq
  \nu$ of probability measures on $X$ by $\nu_k = (1 - p)\mu^0_k +
  p\mu^1_k$. Then for every sequence $\seq r$ of positive numbers,
  \[
  f_{\seq \nu}(\seq r) = \max\left(f_{\seq\mu^0}(\seq r),
  f_{\seq\mu^1}(\seq r)\right).
  \]

  \begin{proof}
    Define the map $\pi \colon \Sigma \times \Omega \times \Omega \to
    \Omega$ by
    \[
    \pi(\sigma, \bm\omega)_k = \omega^{\sigma_k}_k,
    \]
    where $\bm\omega = (\omega^0, \omega^1) \in \Omega \times \Omega$,
    and let $\Q = \P{p} \times \P{\seq \mu^0} \times \P{\seq \mu^1}$.
Then $\P{\seq \nu} = \Q \circ \pi^{-1}$, so it is
    enough to show that
    \[
    \dimh E_{\seq r}(\pi(\sigma, \bm\omega)) =
    \max\left(f_{\seq\mu^0}(\seq r), f_{\seq\mu^1}(\seq r)\right)
    \]
    for $\Q$-a.e.~$(\sigma, \bm\omega)$. For $i = 0, 1$, let
    \[
    A_i(\sigma, \bm\omega) =
    \limsup_{\sigma_k = i} \cball{\omega^i_k}{r_k}.
    \]
    Then for each fixed $\bm\omega$
    \[
    \dimh A_i(\sigma, \bm\omega) = \dimh E_{\seq r} (\omega^i)
    \]
    for $\P{p}$-a.e.~$\sigma$ by Lemma~\ref{posproblemma}, and hence
    \[
    \dimh A_i(\sigma, \bm\omega) = f_{\seq\mu^i}(\seq r)
    \]
    for $\Q$-a.e.~$(\sigma, \bm\omega)$. Since
    \[
    E_{\seq r}(\pi(\sigma, \bm\omega)) =
    A_0(\sigma, \bm\omega) \cup
    A_1(\sigma, \bm\omega)
    \]
    for every $(\sigma, \bm\omega)$, the lemma now follows by the
    finite stability of the Hausdorff dimension.
  \end{proof}
\end{lemma}

\begin{lemma} \label{problemma}
  Let $\seq\mu$ and $\seq\nu$ be sequences of probability measures on
  $X$ and assume that $\mu_k \ll \nu_k$ for every $k$, and that the
  densities $\mathrm{d}\mu_k / \mathrm{d}\nu_k$ are bounded uniformly
  in $k$.  Then $f_{\seq\mu}(\seq r) \leq f_{\seq\nu}(\seq r)$ for
  every $\seq r$.

\begin{proof}
  Let $C > 1$ be such that $\mathrm{d} \mu_k / \mathrm{d}\nu_k < C$
  for every $k$, and define $\widetilde \mu_k$ by
  \[
  \mathrm{d} \widetilde\mu_k = \frac{1}{C - 1}
  \left(C - \frac{\,\mathrm{d}\mu_k}{\, \mathrm{d}\nu_k}\right)
  \, \mathrm{d}\nu_k.
  \]
  Then for every $k$,
  \[
  \nu_k = \frac{1}{C}\mu_k + \frac{C - 1}{C}\widetilde\mu_k,
  \]
  and the lemma follows from Lemma~\ref{convcomblemma}.
\end{proof}
\end{lemma}

Let $Y$ be a measurable space and let $x \mapsto \mu^x$ be a
measure-valued function defined on $X$ and taking values in the set of
probability measures on $Y$. Assume that the function $x \mapsto
\mu^x(A)$ is measurable for every measurable subset $A$ of $Y$, and
let $\nu$ be a probability measure on $X$. Then it is not difficult to
check that the set function
\[
A \mapsto \int \mu^x(A) \intd\nu(x)%,
\]
%defined for measurable subsets $A$ of $Y$,
is a probability measure on $Y$. It will be denoted by $\int \mu^x
\intd\nu(x)$.

If $(\mu^x_k)_{k = 1}^K$ is a finite or countable ($K = \infty$)
sequence of measure-valued functions according to the above, and
$(\nu_k)_{k = 1}^K$ is a sequence of probability measures on $X$, then
\begin{equation}	\label{prodintcommeq}
  \int \bigtimes_k \mu^{x_k} \intd{(\bigtimes_k \nu_k)}(\seq x)
  =
  \bigtimes_k \int \mu_k^x \intd\nu_k(x).
\end{equation}
(the first integral is over $X^K$, and the measures on both sides are
probability measures on $Y^K$). For finite $K$ this is an easy
consequence of Fubini's theorem, and then it follows for $K = \infty$
since the two measures in~\eqref{prodintcommeq} agree on cylinders in
$Y^\nat$.

\begin{lemma}	\label{rholemmma}
  Let $\mu$ and $\nu$ be probability measures on $X$ and let $r$ and
  $C$ be positive numbers such that,
  \[
  \mu\left(\oball{x}{2r}\right) \leq C \nu\left(\oball{x}{r}\right)
  \]
  for every $x \in \supp \mu$. Define the probability measure $\lambda$
  on $X$ by
  \[
  \lambda = \int_{G_r} \mu^x \intd\nu(x),
  \]
  where
  \[
  G_r = \{x \in X; \, \mu(\oball{x}{r}) > 0 \}
  \]
  and
  \[
  \mu^x = \frac{\restr{\mu}{\oball{x}{r}}}{\mu(\oball{x}{r})}.
  \]
  Then $\mu \ll \lambda$ and $\mathrm{d} \mu / \mathrm{d} \lambda \leq
  C$.

\begin{proof}
  It will be shown that $\intd\lambda = \rho \intd\mu$, where $\rho
  \geq C^{-1}$ $\mu$-a.e. Using the definition of $\lambda$,
  \begin{align*}
    \lambda(A) &= \int_{G_r} \frac{1}{\mu(\oball{x}{r})} \int_{\oball{x}{r}}
    \chi_A(y) \intd\mu(y) \intd\nu(x)\\
    &=
    \iint_{|x - y| < r} \chi_A(y) \frac{\chi_{G_r}(x)}{\mu(\oball{x}{r})}
    \intd\nu(x) \intd\mu(y) \\
    &=
    \int \chi_A(y) \int_{\oball{y}{r}}
    \frac{\chi_{G_r}(x)}{\mu(\oball{x}{r})}
    \intd\nu(x) \intd\mu(y)
    =
    \int_A \rho(y) \intd\mu(y),
  \end{align*}
  where
  \[
  \rho(y) =
  \int_{\oball{y}{r}} \frac{\chi_{G_r}(x)}{\mu(\oball{x}{r})} \intd\nu(x).
  \]
  If $y \in \supp \mu$ and $x \in \oball{y}{r}$ then
  $\mu(\oball{x}{r}) > 0$ since $\oball{x}{r}$ is an open set
  containing $y$. This shows that $\oball{y}{r} \subset G_r$ for $y
  \in \supp \mu$, and hence for $\mu$-a.e.~$y$,
  \begin{align*}
    \rho(y) &= \int_{\oball{y}{r}} \frac{1}{\mu(\oball{x}{r})}\intd\nu(x)
    \geq
    \int_{\oball{y}{r}} \frac{1}{\mu(\oball{y}{2r})}\intd\nu(x) \\
    &=
    \frac{\nu(\oball{y}{r})}{\mu(\oball{y}{2r})} \geq \frac{1}{C}.\qedhere
  \end{align*}
\end{proof}
\end{lemma}

\begin{lemma} \label{Clemma}
  Let $\seq\mu$ and $\seq\nu$ be sequences of probability measures on
  $X$ and let $\seq r$ be a sequence of positive numbers. Assume that
  there is a constant $C$ such that for $k = 1, 2, \ldots$ and every
  $x \in \supp\mu_k$,
  \[
  \mu_k(\oball{x}{2r_k}) \leq C \nu_k(\oball{x}{r_k}).
  \]
  Then $f_{\seq\mu}(\seq r) \leq f_{\seq\nu}(2\seq r)$.

  \begin{proof}
    For $k = 1, 2, \ldots$ and $x \in X$ such that
    $\mu_k(\oball{x}{r_k}) > 0$, let
    \[
    \mu_k^x = \frac{\restr{\mu_k}{\oball{x}{r_k}}}{\mu_k(\oball{x}{r_k})}
    \]
    and define the probability measures $\lambda_k$ by
    \[
    \lambda_k = \int_{\{x ;\,\mu_k(\oball{x}{r_k}) > 0\}} \mu_k^x
    \intd\nu_k(x).
    \]
    For every $k$, $\mu_k \ll \lambda_k$ with
    $\mathrm{d}\mu_k/\mathrm{d}\nu_k \leq C$ by Lemma~\ref{rholemmma},
    and hence $f_{\seq \mu}(\seq r) \leq f_{\seq \lambda}(\seq r)$ by
    Lemma~\ref{problemma}. Thus it remains to show that
    $f_{\seq\lambda}(\seq r) \leq f_{\seq \nu}(2 \seq r)$.
    
    Let $t = f_{\seq\lambda}(\seq r)$ and $G = \{\omega \in \Omega; \,
    \dimh E_{2\seq r}(\omega) \geq t\}$.  For $\omega \in \Omega$, let
    $\P{\seq \mu}^\omega = \times_{k = 1}^\infty \mu_k^{\omega_k}$.
    Then the probability measure $\P{\seq \lambda} = \times_{k =
      1}^\infty \lambda_k$ can also be expressed,
    by~\eqref{prodintcommeq}, as
    \[
    \P{\seq\lambda} = \int \P{\seq\mu}^\omega \intd{\P{\seq\nu}}(\omega).
    \]
    For every fixed $\omega \in \Omega$, it holds that $E_{\seq
      r}(\eta) \subset E_{2\seq r}(\omega)$ for
    $\P{\seq\mu}^\omega$-a.e.~$\eta \in \Omega$, and thus
    \[
    \chi_G(\omega) \geq
    \P{\seq\mu}^\omega\{\eta \in \Omega; \, \dimh E_{\seq r}(\eta) \geq t\}.
    \]
    Integrating this gives
    \begin{align*}
      \P{\seq\nu}(G) &\geq \int \P{\seq\mu}^\omega\{\eta \in \Omega;
      \, \dimh E_{\seq r}(\eta) \geq t\}\intd{\P{\seq\nu}}(\omega)
      \\ &= \P{\seq\lambda}(\{\eta \in \Omega; \, \dimh E_{\seq
        r}(\eta) \geq t\}) = 1.\qedhere
    \end{align*}
  \end{proof}
\end{lemma}

\begin{lemma}	\label{finallemma}
  Let $\mu$ and $\theta$ be probability measures on $X$, let $(V_k)_{k
    = 1}^\infty$ be a sequence of Borel subsets of $X$ such that $0 <
  \mu(V_k) < 1$ for every $k$, and let $(r_k)_{k = 1}^\infty$ be a
  decreasing sequence of positive numbers. Define
  \begin{align*}
  p_k &= \mu(V_k), &  
  M_k &= \sum_{i = 1}^k p_i,\\
  k_n &= \min\{k; \, M_k / 2 \geq n\}, &  
  \rho_n &= \frac{1}{2} r_{k_n}, \\
  \mu_k^0 &= \frac{\restr{\mu}{V^c_k}}{\mu(V^c_k)}, &
  \text{and}\qquad 
  \mu_k^1 &= \frac{\restr{\mu}{V_k}}{\mu(V_k)}.
  \end{align*}
  Assume that
  \[
  \lim_{k \to \infty} \frac{M_k}{\log k} = \infty,
  \]
  and that there is a constant $C$ such that for every $k$ and every
  $x \in \supp\theta$,
  \[
  \theta(\oball{x}{2r_k}) \leq C \mu_k^1(\oball{x}{r_k}).
  \]
  Then $f_\mu(\seq r) \geq f_\theta(\seq\rho)$.

  \begin{proof}
    For $\sigma \in \Sigma$, let
    \[
    \P{\mu}^\sigma = \bigtimes_{k = 1}^\infty \mu_k^{\sigma_k},
    \]
    and define the probability measure $\P{\seq p}$ on $\Sigma$ by
    \[
    \P{\seq p} = \bigtimes_{k = 1}^\infty \left( (1 - p_k)\delta_0 +
    p_k\delta_1\right).
    \]
    For every $k$,
    \[
    \mu = (1 - p_k) \mu_k^0 + p_k \mu_k^1,
    \]
    and thus by~\eqref{prodintcommeq},
    \[
    \P{\mu} = \int \P{\mu}^\sigma \intd{\P{\seq p}}(\sigma).
    \]
    Let $t = f_\theta(\seq\rho)$, and define the sets
    \[
    G = \{\omega \in \Omega; \, \dimh E_{\seq r}(\omega) \geq t\}
    \]
    and
    \[
    A = \left\{\sigma \in \Sigma; \, \exists n_0 \text{ s.t.~} \sum_{k
      = 1}^n \sigma_k \geq M_n / 2 \text{ for every } n \geq
    n_0\right\}.
    \]
    For $\sigma \in \Sigma$ and $\omega \in \Omega$, let
    \[
    E_{\seq r}(\sigma, \omega) = \limsup_{\sigma_k = 1}
    \cball{\omega_k}{r_k},
    \]
    and define for $\sigma\in\Sigma$,
    \[
    G_\sigma = \{\omega \in \Omega; \, \dimh E_{\seq r}(\sigma,
    \omega) \geq t \}.
    \]
    Then $G_\sigma \subset G$ for every $\sigma$, and $\P{\seq p}(A)
    = 1$ by Lemma~\ref{deviationlemma}, and hence
    \[
    \P{\mu}(G) = \int \P{\mu}^\sigma(G) \intd{\P{\seq p}}
    \geq \int_A \P{\mu}^\sigma(G_\sigma) \intd{\P{\seq p}}.
    \]
    To conclude the proof, it will be shown that
    $\P{\mu}^\sigma(G_\sigma) = 1$ for every $\sigma \in A$.
    
    So fix $\sigma \in A$, and let $(l_n)_{n = 1}^\infty$ be the
    strictly increasing enumeration of $\{l; \, \sigma_l =
    1\}$. Define $\pi \colon \Omega \to \Omega$ by
    \[
    \pi(\omega)_n = \omega_{l_n}.
    \]
    Then
    \[
    \P{\mu}^\sigma \circ \pi^{-1} = \bigtimes_{n = 1}^\infty \mu_{l_n}^1,
    \]
    and
    \[
    E_{2\seq \rho} (\pi(\omega)) \subset E_{\seq r}(\sigma, \omega)
    \]
since $(r_k)$ is decreasing.
    Thus by Lemma~\ref{Clemma}, $\P{\mu}^\sigma$-a.e.~$\omega$ is such
    that
    \[
    \dimh E_{\seq r}(\sigma, \omega) \geq \dimh E_{2\seq \rho}
    (\pi(\omega)) \geq f_\theta(\seq\rho). \qedhere
    \]
  \end{proof}
\end{lemma}

\subsection{Continuing the example}
In this section, $\mu$ denotes the measure defined in the example and
$C$ is the ternary Cantor set. Recall that for $r \in (0,1)$, the set
$C_r$ is the $r$-fattening of $C$, that is,
\[
C_r = \{x; \, \dist(x, C) \leq r \}.
\]
(Technically, this does not conflict with the notation $C_k$ for the
$k$:th approximation to $C$, since $k$ is an integer and $r$ is
not.)

\begin{lemma} \label{pklemma}
  For every $r \in (0, 1)$,
  \[
  2^{-(\beta - 1)} r^{(\beta - 1)\frac{\log 2}{\log 3}}
  \leq
  \mu(C_r)
  \leq
  2^{2(\beta - 1)} r^{(\beta - 1)\frac{\log 2}{\log 3}}.
  \]

  \begin{proof}
    Fix $r$ and let $n$ be the unique integer such that $3^{-(n +
      1)} \leq r < 3^{-n}$.  Then $C_{n + 1} \subset C_r$
    and $C_r \cap \supp \mu_k = \emptyset$ for $k \leq n -
    2$. Thus
    \[
    (1 - 2^{1 - \beta}) \sum_{k = n + 1}^\infty 2^{(1 - \beta)k}
    \leq \mu(C_r) \leq
    (1 - 2^{1 - \beta}) \sum_{k = n - 1}^\infty 2^{(1 - \beta)k},
    \]
    or equivalently,
    \[
    2^{-(\beta - 1)(n + 1)}
    \leq \mu(C_r) \leq
    2^{-(\beta - 1)(n - 1)}.
    \]
    Using that
    \[
    2^{-n} = 3^{-n \frac{\log 2}{\log 3}} \in \left[r^{\frac{\log
          2}{\log 3}}, \, 2 r^{\frac{\log 2}{\log 3}}\right]
    \]
    gives the stated estimates.
  \end{proof}
\end{lemma}

\begin{exampleagain}[continued]
  Let $\theta$ be the uniform probability measure on $C$. Let $r_k =
  k^{-\alpha}$ and $V_k = C_{r_k}$, and define $p_k$, $M_k$, $k_n$ and
  $\rho_n$ as in Lemma~\ref{finallemma}. By Lemma~\ref{pklemma}, there
  is a constant $c_1$ such that $p_k \geq k^{-\alpha\gamma}$, where
  $\gamma = (\beta - 1)\frac{\log 2}{\log 3} = s_1 - d_1$, and thus
  there is another constant $c_2$ such that $M_k \geq c_2 k^{1 -
    \alpha\gamma}$.  Note that
  \[
  \alpha\gamma = \alpha (s_1 - d_1) \in (0, 1),
  \]
  since $0 < d_1 < s_1$ and it is assumed that $1/\alpha \leq s_1$. In
  particular,
  \[
  \lim_{k \to \infty} \frac{M_k}{\log k} = \infty.
  \]

  Consider an arbitrary point $x \in C$.  It is known that there is a
  constant $c_3$ such that $\theta(\oball{x}{2r_k}) \leq c_3
  r_k^{d_1}$, and it was shown in the first part of the example that
  there is a constant $c_4$ such that $\mu(V_k \cap \oball{x}{r_k}) =
  \mu(\oball{x}{r_k}) \geq c_4 r_k^{s_1}$. By Lemma~\ref{pklemma},
  there is a constant $c_5$ such that $\mu(V_k) \leq c_5 r_k^{s_1 -
    d_1}$. Thus
  \[
  \theta(\oball{x}{2r_k}) \leq c_3 r_k^{d_1}
  = c_3 \frac{r_k^{s_1}}{r_k^{s_1 - d_1}} \leq
  \frac{c_3c_5}{c_4} \frac{\mu(V_k \cap \oball{x}{r_k})}{\mu(V_k)},
  \]
  so that the hypotheses of Lemma~\ref{finallemma} are satisfied (the
  constants $c_3, c_4, c_5$ are independent of $x$ and $k$).

  From $M_k \geq c_2 k^{1 - \alpha\gamma}$ it follows that there is a
  constant $c_6$ such that $k_n \leq c_6 n^{1/(1 - \alpha\gamma)}$,
  and hence there is a constant $c_7$ such that $\rho_n \geq c_7
  n^{-\alpha / (1 - \alpha\gamma)}$. Thus by Lemma~\ref{finallemma}
  and Theorem~\ref{mainthm},
  \[
  f_\mu(\alpha) \geq f_\theta\left(\frac{(1 + \varepsilon)\alpha}{1 -
    \gamma\alpha}\right) = \frac{1 - \gamma\alpha}{(1 +
    \varepsilon)\alpha}
  \]
  for every sufficiently small $\varepsilon > 0$, so that
  \[
  f_\mu(\alpha) \geq \frac{1 - \gamma\alpha}{\alpha} =
  \frac{1}{\alpha} - (s_1 - d_1).
  \]
\end{exampleagain}

\end{document}